\documentclass[oneside,english]{amsart}
\usepackage[T1]{fontenc}
\usepackage[latin9]{inputenc}
\usepackage{mathrsfs}
\usepackage{amstext}
\usepackage{amsthm}
\usepackage{amssymb}
\usepackage[all]{xy}

\makeatletter
\numberwithin{equation}{section}
\numberwithin{figure}{section}
\theoremstyle{plain}
\newtheorem{thm}{\protect\theoremname}
\theoremstyle{definition}
\newtheorem{defn}[thm]{\protect\definitionname}
\theoremstyle{remark}
\newtheorem{rem}[thm]{\protect\remarkname}
\theoremstyle{plain}
\newtheorem{prop}[thm]{\protect\propositionname}
\theoremstyle{plain}
\newtheorem{lem}[thm]{\protect\lemmaname}
\theoremstyle{plain}
\newtheorem{cor}[thm]{\protect\corollaryname}

\usepackage[dvipsnames,svgnames,x11names,hyperref]{xcolor}
\usepackage{lmodern}
\renewcommand*{\epsilon}{\varepsilon}
\usepackage{amssymb}
\usepackage{amsfonts}
\usepackage{egothic}
\usepackage{xcolor}
\usepackage[T1]{fontenc}

\usepackage{url}
\usepackage{amsthm}
\usepackage{aliascnt}
\usepackage{lipsum}
\usepackage{mathrsfs}
\usepackage{esint}
\usepackage{url}
\usepackage{amsmath}
\usepackage{dsfont}
\usepackage{bbm}
\usepackage{pdflscape}
\usepackage{bbm}
\usepackage{bussproofs}

\makeatletter
\def\matrixobject@{%
  \edef \next@{={\DirectionfromtheDirection@ }}%
  \expandafter \toks@ \next@ \plainxy@
  \let\xy@@ix@=\xyq@@toksix@
  \xyFN@ \OBJECT@}
\let\xy@entry@@norm=\entry@@norm
\def\entry@@norm@patched{%
  \let\object@=\matrixobject@
  \xy@entry@@norm }
\AtBeginDocument{\let\entry@@norm\entry@@norm@patched}
\makeatother

\newcommand{\twocong}[2][0.5]{\ar@{}[#2] \save ?(#1)*{\cong}\restore}
\newcommand{\twoeq}[2][0.5]{\ar@{}[#2] \save ?(#1)*{=}\restore}
\newcommand{\ltwocell}[3][0.5]{\ar@{}[#2] \ar@{=>}?(#1)+/r 0.15cm/;?(#1)+/l 0.15cm/^{#3}}
\newcommand{\rtwocell}[3][0.5]{\ar@{}[#2] \ar@{=>}?(#1)+/l 0.15cm/;?(#1)+/r 0.15cm/^{#3}}
\newcommand{\utwocell}[3][0.5]{\ar@{}[#2] \ar@{=>}?(#1)+/d  0.15cm/;?(#1)+/u 0.15cm/_{#3}}
\newcommand{\dtwocell}[3][0.5]{ \ar@{}[#2] { \ar@{=>}?(#1)+/u  0.15cm/;?(#1)+/d 0.15cm/^{#3}}}
\newcommand{\ultwocell}[3][0.5]{\ar@{}[#2] \ar@{=>}?(#1)+/dr  0.15cm/;?(#1)+/ul 0.15cm/^{#3}}
\newcommand{\urtwocell}[3][0.5]{\ar@{}[#2] \ar@{=>}?(#1)+/dl  0.15cm/;?(#1)+/ur 0.15cm/^{#3}}
\newcommand{\dltwocell}[3][0.5]{\ar@{}[#2] \ar@{=>}?(#1)+/ur  0.15cm/;?(#1)+/dl 0.15cm/^{#3}}
\newcommand{\drtwocell}[3][0.5]{\ar@{}[#2] \ar@{=>}?(#1)+/ul  0.15cm/;?(#1)+/dr 0.15cm/^{#3}}

\newcommand{\myar}[2]{\ar^-{#1}[#2]}
\newcommand{\myard}[2]{\ar_-{#1}[#2]}

\usepackage[colorlinks=true,linkcolor=Dark Red, citecolor=Dark Green,linktoc=all]{hyperref}
\usepackage{amsmath}
\usepackage{cleveref}

\makeatletter

  \def\make@df@tag@@#1{%
    \gdef\df@tag{%
      \maketag@@@{\Hy@make@anchor#1}%
      \def\@currentlabel{#1}%
      \def\cref@currentlabel{[equation][2147483647][]#1}%
    }%
  }
  \def\make@df@tag@@@#1{%
    \gdef\df@tag{%
      \tagform@{\Hy@make@anchor#1}%
      \toks@\@xp{\p@equation{#1}}%
      \edef\@currentlabel{\the\toks@}%
      \edef\cref@currentlabel{[equation][2147483647][]\the\toks@}
    }%
  }
\makeatother

\numberwithin{thm}{subsection}

\makeatother

\usepackage{babel}
\providecommand{\corollaryname}{Corollary}
\providecommand{\definitionname}{Definition}
\providecommand{\lemmaname}{Lemma}
\providecommand{\propositionname}{Proposition}
\providecommand{\remarkname}{Remark}
\providecommand{\theoremname}{Theorem}

\begin{document}
\title{Presentations of Pseudodistributive Laws}
\author{Charles Walker}
\address{Department of Mathematics and Statistics, Masaryk University, Kotl{\'a}{\v r}sk{\'a}
2, Brno 61137, Czech Republic}
\email{\tt{charles.walker.math@gmail.com}}
\begin{abstract}
By considering the situation in which the involved pseudomonads are
presented in no-iteration form, we deduce a number of alternative
presentations of pseudodistributive laws including a ``decagon''
form, a pseudoalgebra form, a no-iteration form, and a warping form.
As an application, we show that five coherence axioms suffice in the
usual monoidal definition of a pseudodistributive law.
\end{abstract}

\maketitle
\tableofcontents{}

\section{Introduction}

Monads are one of the fundamental constructions in category theory,
and in recent years have also become more prevalent in computer science
\cite{bisimulation,variablebinding,Hyland2002}. Typically, a monad
on a category $\mathscr{C}$ is defined as an endofunctor $T\colon\mathscr{C}\to\mathscr{C}$
along with natural transformations $u\colon1_{\mathscr{C}}\to T$
and $m\colon T^{2}\to T$ satisfying three coherence conditions.

Distributive laws of monads were introduced by Beck \cite{beckdist}
and give a concise description of the data and coherence conditions
needed to compose two monads $\left(T,u,m\right)$ and $\left(P,\eta,\mu\right)$.
More precisely, Beck defines a distributive law of monads as a natural
transformation $\lambda\colon TP\to PT$ such that the below two triangles
and two pentagons commute
\[
\xymatrix@=1.5em{TP\ar[rr]^{\lambda} &  & PT &  & TP\ar[rr]^{\lambda} &  & PT\\
 & P\ar[ru]_{Pu}\ar[lu]^{uP} &  &  &  & T\ar[ru]_{\eta T}\ar[lu]^{T\eta}\\
T^{2}P\ar[d]_{mP}\ar[r]^{T\lambda} & TPT\ar[r]^{\lambda T} & PTT\ar[d]^{Pm} &  & TP^{2}\ar[d]_{T\mu}\ar[r]^{\lambda P} & PTP\ar[r]^{P\lambda} & P^{2}T\ar[d]^{\mu T}\\
TP\ar[rr]_{\lambda} &  & PT &  & TP\ar[rr]_{\lambda} &  & PT
}
\]
It is not hard to arrive at this set of four axioms. Indeed, given
a $\lambda\colon TP\to PT$ if one works out what is required to extend
the monad $\left(T,u,m\right)$ to a monad 
\[
\left(\widetilde{T},\widetilde{u},\widetilde{m}\right)\colon\mathbf{Kl}\left(P,\eta,\mu\right)\to\mathbf{Kl}\left(P,\eta,\mu\right)
\]
 on the Kleisli category of $P$, they arrive at two of these axioms
from the nullary and binary functoriality conditions of $\widetilde{T}$,
and the other two from naturality of $\widetilde{u}$ and $\widetilde{m}$.

It turns out that one may take a different approach to distributive
laws, based on the ``extensive'' (also called ``no-iteration''
or ``Kleisli triple'') presentation of monads as studied by Manes
\cite{Manes1976}, which in fact dates back to early work of Walters
\cite{WaltersPHD}. In this extensive form, a monad on a category
$\mathscr{C}$ is defined as an assignation on objects $T\colon\mathscr{C}_{\textnormal{ob}}\to\mathscr{C}_{\textnormal{ob}}$
with a family of arrows $u_{X}\colon X\to TX$ and functions $\mathscr{C}\left(X,TY\right)\to\mathscr{C}\left(TX,TY\right)$.
This data is then required to satisfy three different coherence axioms.
It is an interesting fact that the functoriality and naturality conditions
automatically follow from these three conditions\footnote{This simplification which happens in $\mathbf{Cat}$ when monads are
presented extensively is explained in detail by Marmolejo and Wood
\cite{noiteration1d}, as a consequence of any functor having a right
adjoint in the bicategory of profunctors $\mathbf{Prof}.$}. Moreover, this family of functions $\mathscr{C}\left(X,TY\right)\to\mathscr{C}\left(TX,TY\right)$
comprise what is called a ``pasting operator''. Such a pasting operator
of a monad is also called the extension operator of an extension system
\cite{NoiterationMixedAnother,noiteration1d}. 

If one works out what is needed to extend the monad $\left(T,u,m\right)$
to the Kleisli category of $P$, with this extension now defined in
extensive form, they will naturally arrive at three coherence conditions
for distributive laws corresponding to the three axioms for a monad
in extensive form. These three axioms are the two triangles from earlier,
but with the two pentagons replaced by a single decagon condition
\[
\xymatrix@=1em{TPTPT\ar[rr]^{TP\lambda T}\ar[d]_{\lambda TPT} &  & TP^{2}T^{2}\ar[rr]^{TP^{2}m} &  & TP^{2}T\ar[rr]^{T\mu T} &  & TPT\ar[d]^{\lambda T}\\
PT^{2}PT\ar[d]_{PmPT} &  &  &  &  &  & PT^{2}\ar[d]^{Pm}\\
PTPT\ar[rr]_{P\lambda T} &  & P^{2}T^{2}\ar[rr]_{P^{2}m} &  & P^{2}T\ar[rr]_{\mu T} &  & PT
}
\]
In one dimension, the difference between these two definitions of
distributive law is rather trivial. However, in two dimensions the
difference becomes significant, as this means pseudodistributive laws
can be naturally defined taking three modifications as the basic data
rather than the usual four \cite{marm1999}. Moreover, the reduction
in the data makes the coherence conditions for pseudodistributive
laws much easier to understand conceptually. Interestingly, one recovers
a variant of the well known triangle and pentagon axioms for monoidal
categories \cite{MacLane} (though it is actually the bicategorical
version of these axioms \cite{MacPare} which is relevant here).

It is this understanding, along with known coherence results concerning
the triangle and pentagon equations in the context of bicategories
\cite{kellymaclanecoherence} and pseudomonads \cite{marm1997,mndwarp},
that allow us to deduce that three of Marmolejo and Wood's eight coherence
axioms for pseudodistributive laws \cite{marm2008} are redundant
in the sense that they follow from the other five.

However, the goal of this paper is not just to reduce the coherence
axioms of pseudodistributive laws, but to give other presentations
of them. For instance, the reader will notice the composites $\lambda T\cdot Pm\colon TPT\to PT$
appearing in the decagon, so that denoting this composite by $\alpha$,
the decagon may be seen as the hexagon axiom
\[
\xymatrix@=1em{TPTPT\ar[rr]^{TP\alpha}\ar[dd]_{\alpha PT} &  & TP^{2}T\ar[rr]^{T\mu T} &  & TPT\ar[dd]^{\alpha}\\
\\
PTPT\ar[rr]_{P\alpha} &  & P^{2}T\ar[rr]_{\mu T} &  & PT
}
\]

These morphisms $\alpha\colon TPT\to PT$ (or morphisms $PTP\to PT$
in the dual situation\footnote{One might denote such morphisms by $\mathbf{res}_{T\eta}\colon PTP\to PT$
as they exhibit $T\eta$ as a $P$-embedding, using the ``admissibility''
point of view \cite{WalkerDL}. In \cite{WalkerDL} the dual problem
of extending to pseudo-algebras extensively was considered, though
in the simpler lax-idempotent case \cite{zober1976,kock1995}, and
the decagon conditions were not recognized.

When $P$ is only an endofunctor, one is forced to use this dual version.
In particular, a dual version of this hexagon appeared in \cite{Chk}
in the context of wreaths.}) should be familiar to the reader, appearing in the characterization
of distributive laws in terms of Kleisli and Eilenberg-Moore objects
\cite{marm2002}. Indeed, such characterizations are often useful
for considering distributive laws when Kleisli and Eilenberg-Moore
objects do not exist \cite{marm2002}. 

Interestingly, this hexagon axiom leads to a simpler version of these
characterizations of \cite{marm2002}. It turns out that distributive
laws $\lambda$ are in bijection with morphisms $\alpha\colon TPT\to PT$
rendering commutative the diagrams
\begin{equation}
\xymatrix@=1em{TPT\ar[rr]^{\alpha} &  & PT & T^{2}\myar{T\eta T}{rr}\ar[dd]_{m} &  & TPT\ar[dd]^{\alpha} & TPTPT\ar[rr]^{TP\alpha}\ar[dd]_{\alpha PT} &  & TP^{2}T\ar[rr]^{T\mu T} &  & TPT\ar[dd]^{\alpha}\\
\\
PT\ar[uu]^{uPT}\ar[rruu]_{\textnormal{id}} &  &  & T\myard{\eta T}{rr} &  & PT & PTPT\ar[rr]_{P\alpha} &  & P^{2}T\ar[rr]_{\mu T} &  & PT
}
\label{mindef1}
\end{equation}
which we refer to as the algebra definition (as the $T$-algebra axioms
on $\alpha$ follow from the above conditions). This definition is
closely connected to a definition of distributive laws in terms of
pasting operators $\left(-\right)^{\lambda}$ due to Marmolejo and
Wood \cite[Theorem 6.2]{noiteration1d}. Indeed, we get a simplification
of their result, finding that with the monad $P$ defined extensively,
a distributive law is a pasting  operator $\left(-\right)^{\lambda}\colon\mathscr{C}\left(-,PT-\right)\to\mathscr{C}\left(T-,PT-\right)$
such that all $f\colon X\to PTY$, $g\colon Y\to PTZ$ and $k\colon X\to TY$
render commutative the three diagrams
\[
\xymatrix@=1em{X\ar@/^{0pc}/[dd]_{uX}\ar@/^{1pc}/[rrdd]^{f} &  &  &  & TX\ar[dd]_{k^{T}}\ar@/^{1pc}/[rrdd]^{\left(\eta TY\cdot k\right)^{\lambda}} &  &  &  & TX\ar[dd]_{f^{\lambda}}\ar@/^{1pc}/[rrdd]^{\left(\left(g^{\lambda}\right)^{P}f\right)^{\lambda}}\\
 & \; &  &  & \; &  &  &  &  & \;\\
TX\ar[rr]_{f^{\lambda}} &  & PTY &  & TY\myard{\eta TY}{rr} &  & PTY &  & PTY\ar[rr]_{\left(g^{\lambda}\right)^{P}} &  & PTZ
}
\]

Finally we give the so called ``warping'' definition, which describes
distributive laws in terms of the data of the extended monad's resulting
Kleisli category structure. Of the corresponding three conditions
\[
\xymatrix@=1em{\mathscr{C}\left(X,PTY\right)\myard{\left(-\right)_{X,Y}^{\lambda}\cdot\left(-\right)_{X}^{u}}{rrr}\ar@/^{1.3pc}/[rrrrr]^{\textnormal{id}} &  & \; & \mathscr{C}\left(TX,PTY\right)\mathscr{C}\left(X,TX\right)\myard{\circ}{rr} &  & \mathscr{C}\left(X,PTY\right)}
\]
\[
\xymatrix@=1em{\mathscr{C}\left(X,TY\right)\myar{\left(-\right)_{X,TY}^{\eta}}{rrr}\myard{\left(-\right)_{TY}^{\eta}\cdot\left(-\right)_{X,Y}^{T}\quad\quad}{rrrd} &  & \; & \mathscr{C}\left(X,PTY\right)\myar{\left(-\right)_{X,Y}^{\lambda}}{rr} &  & \mathscr{C}\left(TX,PTY\right)\\
 &  &  & \mathscr{C}\left(TY,PTY\right)\mathscr{C}\left(TX,TY\right)\myard{\circ}{rur}
}
\]
\[
\xymatrix@=1em{\mathscr{C}\left(Y,PTZ\right)\mathscr{C}\left(X,PTY\right)\myar{\quad\left(\left(-\right)_{Y,Z}^{\lambda}\right)_{TY,TZ}^{P}\mathscr{C}\left(X,PTY\right)}{rrrr}\ar[d]^{\mathscr{C}\left(Y,PTZ\right)\left(-\right)_{X,Y}^{\lambda}} &  &  &  & \mathscr{C}\left(PTY,PTZ\right)\mathscr{C}\left(X,PTY\right)\ar[d]^{\circ}\\
\mathscr{C}\left(Y,PTZ\right)\mathscr{C}\left(TX,PTY\right)\ar[d]^{\left(\left(-\right)_{Y,Z}^{\lambda}\right)_{TY,TZ}^{P}\mathscr{C}\left(TX,PTY\right)} &  &  & \; & \mathscr{C}\left(X,PTZ\right)\ar[d]^{\left(-\right)_{X,Z}^{\lambda}}\\
\mathscr{C}\left(PTY,PTZ\right)\mathscr{C}\left(TX,PTY\right)\myard{\circ}{rrrr} &  &  &  & \mathscr{C}\left(TX,PTZ\right)
}
\]
the first and third conditions are a unit law and associativity law
for the Kleisli category of the extended monad $\widetilde{T}$, where
the second is slightly stronger in that it gives both the remaining
unit law as well as a compatibility condition.

In this paper we will consider the more general 2-dimensional ``pseudo''
versions of the above to better understand pseudodistributive laws.
This allows us to address some problems which have not been practical
to solve until now (due to the complicated coherence conditions).
Perhaps most important is defining pseudodistributive laws in terms
of pasting operators $\left(-\right)^{\lambda}$, which is essential
in the setting of relative pseudomonads \cite{relative}, where one
is forced to use pasting operators. This is perhaps the simplest and
most natural presentation\footnote{One might argue the warping definition is the most natural, since
it is based entirely on (pseudo)pasting operators. However, this presentation
is more complicated to state.}.

It remains unclear how to generalize these results to distributive
laws where both monads are relative. Solving this problem will likely
involve some combination of the methods of this paper, the case when
one is relative \cite{relativedist}, and left and right extension
operators \cite{NoiterationMixedAnother}.

\subsection{Structure of the paper}

In Section \ref{psm} we recall the monoidal and extensive (no-iteration)
definitions of pseudomonads, and their coherence axioms. Then in Section
\ref{2dcase} we give five new presentations of pseudodistributive
laws $\lambda\colon TP\to PT$ of pseudomonads; namely:
\begin{enumerate}
\item the ``pseudomonoidal'' definition. This is the pseudo version of
the usual definition due to Beck \cite{beckdist}, involving a pseudonatural
transformation $\lambda\colon TP\to PT$ and four invertible modifications
comprising two triangles and two pentagons satisfying five coherence
axioms;
\item the ``Kleisli-decagon'' definition. This involves the decagonal
conditions one finds for distributive laws when the involved monads
are presented in extensive form and the usual $\lambda\colon TP\to PT$
is taken as the data. This version comprises three modifications satisfying
a version of the triangle and pentagon equations, plus a third compatibility
axiom\footnote{The ``Kleisli'' prefix refers to the fact the decagon starts with
$TPTPT$, as happens when one extends to the Kleisli category extensively.
There is also a dual version starting from $PTPTP$.};
\item the ``pseudoalgebra'' definition in terms of maps $\alpha\colon TPT\to PT$.
This is a reduced version of the above in which a change of variables
leads to a simplification in the axioms. Moreover, this definition
may be regarded as a ``base case'' for definitions of pseudodistributive
laws in terms of pasting operators $\left(-\right)^{\lambda}$ as
one may apply such pasting operators to identities to recover $\alpha$;
\item the ``no-iteration'' definition in terms of pseudo-pasting operators
\[
\left(-\right)^{\lambda}\colon\mathscr{C}\left(-,PT-\right)\to\mathscr{C}\left(T-,PT-\right).
\]
This no-iteration definition is intended to avoid any iteration of
the involved pseudomonads $T$ and $P$, which is important in the
``relative'' case \cite{relative};
\item the ``warping'' definition in terms of the data of the Kleisli bicategory
of the extended pseudomonad $\widetilde{T}$. This formulation allows
for applications of MacLane and Paré's coherence theorem \cite{MacPare},
and hence applications to the corresponding pseudodistributive law
data in the earlier formulations. Moreover, this generalizes Street
and Lack's equivalence of monads and warpings \cite{mndwarp} to one
of distributive laws and ``distributivity warpings''.
\end{enumerate}
The third and forth presentations above give an improvement of a result
of Marmolejo and Wood \cite[Theorem 6.2]{noiteration1d}, both simplifying
and generalizing from one to two dimensions. In fact, understanding
this result was the original motivation for this paper. 

In Section \ref{2dproof} we justify our five definitions of pseudodistributive
law by proving they are in equivalence with compatible extensions
of $T$ to the Kleisli bicategory of the pseudomonad $P$. In the
case of the pseudomonoidal definition, we will also explain in Subsection
\ref{showred} how one deduces the redundancy of three of the usual
pseudodistributive law axioms from the decagon formulation's version
of the triangle and pentagon equations.

\section{Acknowledgments}

The author would like to thank the members of the Masaryk University
Algebra Seminar, CT20-21 conference, and Open House on Category Theory
2021 for their questions and comments. We also thank Nathanael Arkor
for correcting some typos and his version of the middle no-iteration
axiom in dimension one. Finally we are grateful to the anonymous referee
for their useful feedback and suggestions which led to a number of
improvements in this paper and especially the last section.

\section{Pseudomonads\label{psm}}

We start this section by recalling two equivalent definitions of pseudomonad
(namely the monoidal and no-iteration forms), including the three
axioms which are known to be redundant in monoidal form by results
of Marmolejo \cite{marm1997}, and in no-iteration form (also called
extensive form) by results of Street and Lack \cite{mndwarp}.

\subsection{Pseudomonads in pseudomonoidal and no-iteration form}

In order to define pseudomonads, we first need the notions of pseudonatural
transformations and modifications. The notion of pseudonatural transformation
is the (weak) 2-categorical version of natural transformation. Modifications,
also defined below, take the place of morphisms between pseudonatural
transformations.
\begin{defn}
A \emph{pseudonatural transformation} of pseudofunctors $t\colon F\Rightarrow G\colon\mathscr{A}\to\mathscr{B}$
where $\mathscr{A}$ and $\mathscr{B}$ are bicategories provides
for each 1-cell $f\colon\mathcal{A}\to\mathcal{B}$ in $\mathscr{A}$,
1-cells $t_{\mathcal{A}}$ and $t_{\mathcal{B}}$ and an invertible
2-cell $t_{f}$ in $\mathscr{B}$ as below
\[
\xymatrix@=1em{F\mathcal{A}\ar[rr]^{Ff}\ar[dd]_{t_{\mathcal{A}}} &  & F\mathcal{B}\ar[dd]^{t_{\mathcal{B}}}\\
 & \ar@{}[]|-{\overset{t_{f}}{\implies}}\\
G\mathcal{A}\ar[rr]_{Gf} &  & G\mathcal{B}
}
\]
coherent with composition as in \cite[Definition 2.2]{kelly1974}.
Given two pseudonatural transformations $t,s\colon F\Rightarrow G\colon\mathscr{A}\to\mathscr{B}$
as above, a \emph{modification} $\alpha\colon s\Rrightarrow t$ consists
of, for every object $\mathcal{A}\in\mathscr{A}$, a 2-cell $\alpha_{\mathcal{A}}\colon t_{\mathcal{A}}\Rightarrow s_{\mathcal{A}}$
such that for each 1-cell $f\colon\mathcal{A}\to\mathcal{B}$ in $\mathscr{A}$
we have the equality $\left(\alpha_{\mathcal{B}}\cdot Ff\right)\circ t_{f}=s_{f}\circ\left(Gf\cdot\alpha_{\mathcal{A}}\right)$.
\end{defn}

By considering pseudomonads as pseudomonoids in a Gray-monoid of endo-pseudofunctors
one naturally arrives at the following definition.
\begin{defn}
A \emph{pseudomonad (in pseudomonoidal form)} on a bicategory $\mathscr{C}$
consists of a pseudofunctor equipped with pseudonatural transformations
as below
\[
T\colon\mathscr{C}\to\mathscr{C},\qquad u\colon1_{\mathscr{C}}\to T,\qquad m\colon T^{2}\to T
\]
along with three invertible modifications
\[
\xymatrix@=1em{T\ar[rr]^{uT}\ar[rdrd]_{\textnormal{id}} &  & T^{2}\ar[dd]^{m} &  & T\ar[ll]_{Tu}\ar[ldld]^{\textnormal{id}} &  & T^{3}\ar[rr]^{Tm}\ar[dd]_{mT} &  & T^{2}\ar[dd]^{m}\\
 & \;\ar@{}[ru]|-{\overset{\alpha}{\Longleftarrow}_{\;}} &  & \ar@{}[lu]|-{\overset{\beta}{\Longleftarrow_{\;}}}\\
 &  & T &  &  &  & T^{2}\ar[rr]_{m} &  & T\ar@{}[lulu]|-{\overset{\gamma}{\Longleftarrow}}
}
\]
subject to the two coherence axioms
\[
\xymatrix@=1em{ &  &  &  & T^{2}\ar[rrd]^{m} &  &  &  &  &  & T^{3}\ar[rrd]^{Tm}\\
T^{2}\ar[rr]^{TuT} &  & T^{3}\ar[rru]^{Tm}\ar[rrd]_{mT} &  & \ar@{}[]|-{\Downarrow\gamma} &  & T & = & T^{2}\ar[rur]^{TuT}\ar[rrd]_{TuT}\ar[rrrr]|-{\textnormal{id}} &  & \ar@{}[u]|-{\Downarrow T\alpha}\ar@{}[d]|-{\Downarrow\beta T} &  & T^{2}\ar[rr]^{m} &  & T\\
 &  &  &  & T^{2}\ar[rru]_{m} &  &  &  &  &  & T^{3}\ar[urr]_{mT}
}
\]
\[
\xymatrix@=1em{T^{4}\ar[rr]^{T^{2}m}\ar[dd]_{mT^{2}}\ar[rdrd]^{TmT} &  & T^{3}\ar[rrdd]^{Tm} &  &  &  & T^{4}\ar[rr]^{T^{2}m}\ar[dd]_{mT^{2}} &  & T^{3}\ar[rrdd]^{Tm}\ar[dd]^{mT}\\
 &  & \ar@{}[]|-{\quad\overset{T\gamma}{\Longleftarrow}} &  &  &  &  & \ar@{}[]|-{\overset{m_{m}^{-1}}{\Longleftarrow}}\\
T^{3}\ar[rrdd]_{mT} & \ar@{}[]|-{\overset{\gamma T}{\Longleftarrow}} & T^{3}\ar[rr]^{Tm}\ar[dd]_{mT} &  & T^{2}\ar[dd]^{m} & = & T^{3}\ar[rrdd]_{mT}\ar[rr]_{Tm} &  & T^{2}\ar[rrdd]^{m} & \ar@{}[]|-{\overset{\gamma}{\Longleftarrow}} & T^{2}\ar[dd]^{m}\\
 &  &  & \ar@{}[]|-{\overset{\gamma}{\Longleftarrow}} &  &  &  &  & \ar@{}[]|-{\overset{\gamma}{\Longleftarrow}}\\
 &  & T^{2}\ar[rr]_{m} &  & T &  &  &  & T^{2}\ar[rr]_{m} &  & T
}
\]
\end{defn}

\begin{rem}
\label{consequences} One should note here that there are three useful
consequences of these pseudomonad axioms \cite[Proposition 8.1]{marm1997}
motivated by results of Kelly \cite{kellymaclanecoherence}. These
are as follows:
\[
\xymatrix@=1em{ &  &  &  & T^{2}\ar[rrd]^{m} &  &  &  &  &  & T\ar[rrd]^{uT}\\
1_{\mathscr{C}}\ar[rr]^{u} &  & T\ar[rru]^{uT}\ar[rrd]_{Tu}\ar[rrrr]|-{\textnormal{id}} &  & \ar@{}[u]|-{\Downarrow\alpha}\ar@{}[d]|-{\Downarrow\beta} &  & T & = & 1_{\mathscr{C}}\ar[rru]^{u}\ar[rrd]_{u} &  & \ar@{}[]|-{\Downarrow u_{u}^{-1}} &  & T^{2}\ar[rr]^{m} &  & T\\
 &  &  &  & T^{2}\ar[rru]_{m} &  &  &  &  &  & T\ar[rru]_{Tu}
}
\]
\[
\xymatrix@=1em{T^{2}\ar[rr]^{uT^{2}}\ar[rrdd]_{\textnormal{id}} &  & T^{3}\ar[rr]^{Tm}\ar[dd]_{mT}\ar@{}[dl]|-{\overset{\alpha T}{\Longleftarrow}} &  & T^{2}\ar[dd]^{m} &  &  &  & T^{3}\ar[rr]^{Tm}\ar@{}[rdrd]|-{\Downarrow u_{m}\;} &  & T^{2}\ar[rr]^{m}\ar@{}[dr]|-{\overset{\alpha}{\Longleftarrow}} &  & T\\
 & \; &  & \ar@{}[]|-{\overset{\gamma}{\Longleftarrow}} &  &  & = &  &  &  &  & \;\\
 &  & T^{2}\ar[rr]_{m} &  & T &  &  &  & T^{2}\ar[uu]^{uT^{2}}\ar[rr]_{m} &  & T\ar[uu]^{uT}\ar[ruru]_{\textnormal{id}}
}
\]
\[
\xymatrix@=1em{T^{2}\ar[rr]^{T^{2}u}\ar[rrdd]_{\textnormal{id}} &  & T^{3}\ar[rr]^{mT}\ar[dd]_{Tm}\ar@{}[dl]|-{\overset{T\beta^{-1}}{\Longleftarrow}} &  & T^{2}\ar[dd]^{m} &  &  &  & T^{3}\ar[rr]^{mT}\ar@{}[rdrd]|-{\Downarrow m_{u}^{-1}\;\;} &  & T^{2}\ar[rr]^{m}\ar@{}[dr]|-{\overset{\beta^{-1}}{\Longleftarrow}} &  & T\\
 & \; &  & \ar@{}[]|-{\overset{\gamma^{-1}}{\Longleftarrow}} &  &  & = &  &  &  &  & \;\\
 &  & T^{2}\ar[rr]_{m} &  & T &  &  &  & T^{2}\ar[uu]^{T^{2}u}\ar[rr]_{m} &  & T\ar[uu]^{Tu}\ar[ruru]_{\textnormal{id}}
}
\]

We only mention these redundant axioms as they will be important later
on. Indeed, a version of these three redundant axioms appear in the
coherence conditions for a pseudodistributive law, albeit in a very
convoluted way which is why it was not noticed earlier. As we will
later see in Section \ref{showred}, the first appears most directly,
whilst the appearance of the other two only becomes apparent when
one combines the two pentagons into a decagon.
\end{rem}

The definition of a pseudomonad in no-iteration form is due to Marmolejo
and Wood \cite{NoIteration}. However, it will be more convenient
to use the presentation given by Fiore, Gambino, Hyland and Winskel
\cite{relative} for relative pseudomonads (with the ``relative''
part taken to be an identity).
\begin{defn}
\label{pseudomonadextensive}\cite{NoIteration,relative} A \emph{pseudomonad
(in no-iteration form)} on a bicategory $\mathscr{C}$ consists of:
\begin{itemize}
\item an assignation on objects $\mathscr{C}_{\textnormal{ob}}\to\mathscr{C}_{\textnormal{ob}}\colon X\mapsto TX$;
\item for each $X\in\mathscr{C}$, a 1-cell $u_{X}\colon X\to TX$;
\item for each $X,Y\in\mathscr{C}$ a functor $\left(-\right)_{X,Y}^{T}\colon\mathscr{C}\left(X,TY\right)\to\mathscr{C}\left(TX,TY\right)$;
\item for each $f\colon X\to TY$, an isomorphism $\phi_{f}\colon f\Rightarrow f^{T}\cdot u_{X}$
natural in $f$;
\item for each $X\in\mathscr{C}$, an isomorphism $\theta_{X}\colon u_{X}^{T}\Rightarrow\textnormal{id}_{TX}$;
\item for each $f\colon X\to TY$ and $g\colon Y\to TZ$, an isomorphism
$\delta_{g,f}\colon\left(g^{T}\cdot f\right)^{T}\Rightarrow g^{T}\cdot f^{T}$
natural in $f$ and $g$;
\end{itemize}
satisfying the two coherence conditions involving the unitors and
associators of $\mathscr{C}$:
\begin{enumerate}
\item each $f\colon X\to TY$ renders commutative
\[
\xymatrix@=1em{f^{T}\myar{\phi_{f}^{T}}{rr}\ar@/_{1pc}/[rrrrdd]_{\textnormal{unitor}} &  & \left(f^{T}u_{X}\right)^{T}\myar{\delta_{f,u_{X}}}{rr} &  & f^{T}u_{X}^{T}\ar[dd]^{f^{T}\theta_{X}}\\
\\
 &  &  &  & f^{T}\cdot\textnormal{id}
}
\]
\item each $f\colon X\to TY$, $g\colon Y\to TZ$, and $h\colon Z\to TV$
renders commutative
\[
\xymatrix@=1em{ &  & \left(\left(h^{T}g\right)^{T}f\right)^{T}\ar[rd]^{\delta_{h^{T}g,f}}\ar[dl]_{\left(\delta_{h,g}f\right)^{T}}\\
 & \left(\left(h^{T}g^{T}\right)f\right)^{T}\ar[d]_{\textnormal{assoc.}} &  & \left(h^{T}g\right)^{T}f^{T}\ar[d]^{\delta_{h,g}f^{T}}\\
 & \left(h^{T}\left(g^{T}f\right)\right)^{T}\ar[d]_{\delta_{h^{T},g^{T}f}} &  & \left(h^{T}g^{T}\right)f^{T}\ar[d]^{\textnormal{assoc.}}\\
 & h^{T}\left(g^{T}f\right)^{T}\ar[rr]_{h^{T}\delta_{g,f}} &  & h^{T}\left(g^{T}f^{T}\right)
}
\]
\end{enumerate}
\end{defn}

\begin{rem}
\label{redundantremark} The three useful consequences of the pseudomonad
axioms listed earlier in Remark \ref{consequences} now become the
assertion \cite[Lemma 3.2]{relative} which states that any morphisms
$f\colon X\to TY$ and $g\colon Y\to TZ$ render commutative
\[
\xymatrix@=1em{u_{X}\ar@/_{1pc}/[rrdd]_{\textnormal{id}}\myar{\phi_{u_{X}}}{rr} &  & u_{X}^{T}u_{X}\ar[dd]^{\theta_{X}u_{X}} &  & \left(u_{Y}^{T}f\right)^{T}\myar{\delta_{u_{Y},f}}{rr}\ar@/_{1pc}/[rrdd]_{\left(\theta_{Y}f\right)^{T}} &  & u_{Y}^{T}f^{T}\ar[dd]^{\theta_{Y}f^{T}} &  & g^{T}f\ar[rrdd]_{g^{T}\phi_{f}}\myar{\phi_{g^{T}f}}{rr} &  & \left(g^{T}f\right)^{T}u_{X}\ar[dd]^{\delta_{g,f}u_{X}}\\
\\
 &  & u_{X} &  &  &  & f^{T} &  &  &  & g^{T}f^{T}u_{X}
}
\]
The redundancy of these pseudomonad axioms in no-iteration form was
first noticed by Street and Lack \cite{mndwarp}.

Given that these two versions of pseudomonads are in equivalence,
it is not at all surprising that there are three corresponding redundant
axioms in the two definitions. What is surprising (and is shown later)
is that this causes three of the usual pseudodistributive law axioms
to be redundant. It is not at all expected that the redundant three
of five pseudomonad axioms should correspond to three of the ten pseudodistributive
law axioms (as these are a very different looking sets of axioms).
In fact, this is the best situation one might hope for. 
\end{rem}

$\mathbf{Convention}$: Throughout the remainder of this paper we
will suppress the modification data when describing pseudomonads.
Thus instead of $\left(T,u,m,\alpha,\beta,\gamma\right)$ we simply
write $\left(T,u,m\right)$ or $T$. We will also suppress the modification
data for the pseudomonad $P$, and the isomorphisms in naturality
squares. The reason for this is the proof carries through the same
whether $T$ and $P$ are pseudomonads or 2-monads. It is only the
modification data on the extended pseudomonad that is important for
proving the results of this paper. Moreover, labeling all of this
data may be distracting or overwhelming for the reader, especially
in the larger diagrams which involve higher order expressions such
as $TPTPTPT$. 

We will also follow the convention of Street and Lack \cite{StreetFTM2},
where a distributive law of type $\lambda\colon TP\to PT$ is one
of $T$ over $P$.
\begin{defn}
\label{Kleisli2d} The \emph{Kleisli bicategory} of a pseudomonad
$\left(T,u,m\right)$ on a 2-category $\mathscr{C}$, denoted $\mathbf{Kl}\left(T\right)$,
is the bicategory with:
\begin{itemize}
\item the same objects as $\mathscr{C}$;
\item a morphism $f\colon X\rightsquigarrow Y$ in $\mathbf{Kl}\left(T\right)$
is a morphism $f\colon X\to TY$ in $\mathscr{C}$; 
\item the identity $\textnormal{id}_{X}\colon X\rightsquigarrow X$ on a
object $X$ is the unit $uX\colon X\to TX$;
\item for each $f\colon X\rightsquigarrow Y$ and $g\colon Y\rightsquigarrow Z$
the composite $g\cdot f\colon X\rightsquigarrow Z$ is given by
\[
\xymatrix{X\ar[r]^{f} & TY\myar{Tg}{r} & T^{2}Z\ar[r]^{mZ} & TZ}
.
\]
\end{itemize}
The unitality and associativity laws only hold up to coherent isomorphism,
thus giving only a bicategory. The coherence data for these laws is
constructed using the modifications comprising the pseudomonad. We
will omit the details here, as the reader can refer to \cite[Definition 4.1]{cheng2003}.
\end{defn}

\begin{rem}
The reader will notice the composite above may be written as $g^{T}\cdot f$
when the pseudomonad is presented in no-iteration form. As this formulation
of pseudomonad naturally lends itself to describing the Kleisli bicategory,
it is also sometimes called the Kleisli presentation.

Following this line of reasoning, we observe that the basic structure
required to form the Kleisli bicategory is a (pseudo) pasting operator
$\left(-\right)^{T}\colon\mathscr{C}\left(X,TY\right)\to\mathscr{C}\left(TX,TY\right)$
along with some data and axioms. If the reader is being consistent,
they will express all of the data and axioms in terms of pasting operators.
This gives the (unsimplified) \emph{warping} definition of a pseudomonad.
\end{rem}

\section{Presentations of pseudodistributive laws\label{2dcase}}

\subsection{Pseudomonoidal definition of pseudodistributive laws}

Even when dealing with strict 2-monads, it is often the case that
one has no strict distributive law between them, but only a \emph{pseudo}distributive
law where the usual diagrams (two triangles and two pentagons) only
commute up to invertible modifications \cite{cheng2003}. Work on
these ``pseudo'' versions of distributive laws started with Kelly
\cite{kellydist}, who considered the case where the usual axioms
held strictly with the exception of one of the pentagons. 

Later, pseudodistributive laws were considered in the general case
(where all four axioms only hold up to isomorphism) by Marmolejo \cite{marm1999},
who imposed nine coherence conditions on the four invertible modifications.
It was then later shown by Marmolejo and Wood \cite{marm2008}, that
one of the original nine axioms, in addition to a tenth axiom introduced
by Tanaka \cite{thesisTanaka}, are redundant, thus reducing the number
of coherence axioms to eight. 

We now give another reduction in the coherence axioms, using just
five to define a pseudodistributive law. We again follow the convention
of suppressing the modification data of the pseudomonads $T$ and
$P$.
\begin{rem}
Note that the usual ten coherence axioms for a pseudodistributive
law come from understanding the structure of a pseudomonad in the
Gray category of pseudomonads \cite{GambinoDL}. From there, it is
a matter of working out which are redundant in that they follow from
the others\footnote{An exception to this is in the lax idempotent setting, where one can
define pseudodistributive laws using different sets of coherence axioms.
However, it appears unlikely any choice would be better than the five
axioms given here in the general case.}.
\end{rem}

\begin{rem}
To explain why the ``pseudo'' version of distributive laws is interesting
to consider, we point the reader to the trivial fact that a strict
law is merely a strict monad morphism and opmorphism. This fact is
somewhat misleading as it fails in the weak (pseudo) setting where
an extra coherence axiom is required, as detailed in Theorem \ref{onlyW5}.
Hence this weak case motivates us to stop viewing distributive laws
as a monad morphism and opmorphism, but rather to think of them as
their own independent structure; namely, as a generalization of monads.
This is especially evident in the no-iteration setting, where the
data and axioms for a distributive law are similar to that of a monad.
This also follows Beck \cite{beckdist} who generalized the concept
of algebras from monads to distributive laws. These are ultimately
the algebras of the composite monad but with a simpler description
in terms of the distributive law data.
\end{rem}

\begin{defn}
A \emph{pseudodistributive law (in pseudomonoidal form)} of pseudomonads
$\left(T,u,m\right)$ over $\left(P,\eta,\mu\right)$ is a pseudonatural
transformation $\lambda\colon TP\to PT$ and four invertible modifications
as below\footnote{The directions of the modifications below are chosen such that they
will naturally compose into decagons later on, and such that the directions
of the induced pseudomonad modifications will match with that of a
no-iteration pseudomonad as in Definition \ref{pseudomonadextensive}
(which is defined as in \cite{relative}). Though these choices of
directions do not matter in the sense that these modifications are
invertible, it will make the later proofs easier to follow.}
\[
\xymatrix@=1.5em{TP\ar[rr]^{\lambda} & \; & PT &  &  & TP\ar[rr]^{\lambda} & \; & PT\\
 & P\ar[ru]_{Pu}\ar[lu]^{uP}\ultwocell[0.65]{u}{\omega_{1}} &  &  &  &  & T\ar[ru]_{\eta T}\ar[lu]^{T\eta}\drtwocell[0.6]{u}{\omega_{2}}\\
T^{2}P\ar[d]_{mP}\ar[r]^{T\lambda} & TPT\ar[r]^{\lambda T}\dltwocell[0.5]{d}{\omega_{3}} & PTT\ar[d]^{Pm} &  &  & TP^{2}\ar[d]_{T\mu}\ar[r]^{\lambda P} & PTP\ar[r]^{P\lambda}\urtwocell[0.5]{d}{\omega_{4}} & P^{2}T\ar[d]^{\mu T}\\
TP\ar[rr]_{\lambda} & \; & PT &  &  & TP\ar[rr]_{\lambda} & \;\; & PT
}
\]
satisfying the following five coherence axioms. The first two coherence
axioms are the unitality axioms of a pseudomonad morphism and pseudomonad
opmorphism

\begin{equation}
\tag{{W1}}\xymatrix@=1.5em{TP\ar[rr]^{TuP}\urtwocell[0.6]{rdr}{T\omega_{1}}\ar[d]_{\textnormal{id}}\ar@/^{1.7pc}/[rrr]^{\textnormal{id}} &  & T^{2}P\ar[r]^{mP}\ar[d]_{T\lambda} & TP\ar[dd]^{\lambda}\\
TP\ar[rr]_{TPu}\ar[d]_{\lambda} &  & TPT\urtwocell[0.7]{r}{\omega_{3}}\ar[d]_{\lambda T} & \; & \ar@{}[]|-{=} & \textnormal{id}_{\lambda}\\
PT\ar[rr]_{PTu}\ar@/_{1.7pc}/[rrr]_{\textnormal{id}} &  & PT^{2}\ar[r]_{Pm} & PT
}
\label{W1}
\end{equation}
\begin{equation}
\tag{{W2}}\xymatrix@=1.5em{TP\ar[rr]^{T\eta P}\ar[d]_{\textnormal{id}}\ar@/^{1.7pc}/[rrr]^{\textnormal{id}}\dltwocell[0.4]{rdr}{\omega_{2}P} &  & TP^{2}\ar[r]^{T\mu}\ar[d]_{\lambda P} & TP\ar[dd]^{\lambda}\\
TP\ar[rr]_{\eta TP}\ar[d]_{\lambda} & \; & PTP\ar[d]_{P\lambda}\dltwocell[0.5]{r}{\omega_{4}} & \; & \ar@{}[]|-{=} & \textnormal{id}_{\lambda}\\
PT\ar[rr]_{\eta PT}\ar@/_{1.7pc}/[rrr]_{\textnormal{id}} &  & P^{2}T\ar[r]_{\mu T} & PT
}
\label{W2}
\end{equation}
The next two axioms are the associativity axioms of a pseudomonad
morphism and pseudomonad opmorphism
\begin{equation}
\tag{{W3}}\xymatrix@=1.5em{ & T^{3}P\ar[d]_{TmP}\ar[r]^{T^{2}\lambda}\ar@/_{0.5pc}/[ld]_{mTP} & T^{2}PT\ar[r]^{T\lambda T}\dltwocell[0.5]{d}{T\omega_{3}} & TPT^{2}\ar[d]^{TPm}\myar{\lambda T^{2}}{r} & PT^{3}\ar[d]^{PTm}\\
T^{2}P\ar@/_{0.5pc}/[rd]_{mP} & T^{2}P\ar[d]_{mP}\ar[rr]_{T\lambda} & \; & TPT\ar[r]_{\lambda T}\dltwocell[0.5]{dl}{\omega_{3}} & PT^{2}\ar[d]^{Pm}\\
 & TP\ar[rrr]_{\lambda} & \; & \; & PT\\
 &  & \ar@{}[r]|-{=} & \;\\
 & T^{3}P\myar{T^{2}\lambda}{r}\ar[d]_{mTP} & T^{2}PT\ar[r]^{T\lambda T}\ar[d]_{mPT} & TPT^{2}\ar[r]^{\lambda T^{2}}\dltwocell[0.5]{d}{\omega_{3}T} & PT^{3}\ar[d]^{PmT}\ar@/^{0.5pc}/[rd]^{PTm}\\
 & T^{2}P\ar[d]_{mP}\ar[r]_{T\lambda} & TPT\ar[rr]_{\lambda T}\dltwocell[0.5]{dr}{\omega_{3}} & \; & PT^{2}\ar[d]^{Pm} & PT^{2}\ar@/^{0.5pc}/[dl]^{Pm}\\
 & TP\ar[rrr]_{\lambda} & \; &  & PT
}
\label{W3}
\end{equation}
\begin{equation}
\tag{{W4}}\xymatrix@=1.5em{ & TP^{3}\ar[d]_{T\mu P}\ar[r]^{\lambda P^{2}}\ar@/_{0.5pc}/[ld]_{TP\mu} & PTP^{2}\ar[r]^{P\lambda P}\urtwocell[0.6]{d}{\omega_{4}P} & P^{2}TP\ar[d]^{\mu TP}\myar{P^{2}\lambda}{r} & P^{3}T\ar[d]^{\mu PT}\\
TP^{2}\ar@/_{0.5pc}/[rd]_{T\mu} & TP^{2}\ar[d]_{T\mu}\ar[rr]_{\lambda P} & \; & PTP\ar[r]_{P\lambda}\urtwocell[0.55]{dl}{\omega_{4}} & P^{2}T\ar[d]^{\mu T}\\
 & TP\ar[rrr]_{\lambda} & \; & \; & PT\\
 &  & \ar@{}[r]|-{=} & \;\;\\
 & TP^{3}\myar{\lambda P^{2}}{r}\ar[d]_{TP\mu} & PTP^{2}\ar[r]^{P\lambda P}\ar[d]_{PT\mu} & P^{2}TP\myar{P^{2}\lambda}{r}\urtwocell[0.6]{d}{P\omega_{4}} & P^{3}T\ar[d]^{P\mu T}\ar@/^{0.5pc}/[rd]^{\mu PT}\\
 & TP^{2}\ar[d]_{T\mu}\ar[r]_{\lambda P} & PTP\ar[rr]_{P\lambda}\urtwocell[0.55]{dr}{\omega_{4}} & \; & P^{2}T\ar[d]^{\mu T} & P^{2}T\ar@/^{0.5pc}/[dl]^{\mu T}\\
 & TP\ar[rrr]_{\lambda} & \; &  & PT
}
\label{W4}
\end{equation}
The last axiom ensures that the pentagons $\omega_{3}$ and $\omega_{4}$
are compatible, and asks
\begin{equation}
\tag{{W5}}\xymatrix@=1.5em{ &  & T^{2}P\ar@/^{0pc}/[rr]^{T\lambda}\dtwocell[0.5]{d}{T\omega_{4}} &  & TPT\ar[dr]^{\lambda T}\\
 &  & TP^{2}T\ar[rd]^{\lambda PT}\ar[rur]_{T\mu T} &  &  & PT^{2}\ar[rd]^{Pm}\\
T^{2}P^{2}\ar[r]^{T\lambda P}\ar[urur]^{T^{2}\mu}\ar[rdrd]_{mP^{2}} & TPTP\ar[ur]^{TP\lambda}\ar[rd]_{\lambda TP} &  & PTPT\ar[r]^{P\lambda T}\dtwocell[0.5]{ruu}{\omega_{4}T}\dtwocell[0.5]{rdd}{P\omega_{3}} & P^{2}T^{2}\ar[ru]_{\mu T^{2}}\ar[dr]^{P^{2}m} &  & PT\\
 &  & PT^{2}P\ar[ur]_{PT\lambda}\ar@/_{0pc}/[rrd]^{PmP} &  &  & P^{2}T\ar[ru]_{\mu T}\\
 &  & TP^{2}\ar[rr]_{\lambda P}\dtwocell[0.5]{u}{\omega_{3}P} &  & PTP\ar[ru]_{P\lambda}
}
\label{W5}
\end{equation}
is equal to
\[
\xymatrix@=1.5em{ & T^{2}P\ar[rd]^{mP}\ar[r]^{T\lambda} & TPT\ar[r]^{\lambda T} & PT^{2}\ar[rd]^{Pm}\dtwocell[0.5]{ld}{\omega_{3}}\\
T^{2}P^{2}\ar[ur]^{T^{2}\mu}\ar[rd]_{mP^{2}} &  & TP\ar[rr]^{\lambda} &  & PT\\
 & TP^{2}\ar[ur]_{T\mu}\ar[r]_{\lambda P} & PTP\ar[r]_{P\lambda} & P^{2}T\ar[ur]_{\mu T}\dtwocell[0.5]{lu}{\omega_{4}}
}
\]
\end{defn}

\begin{rem}
Since the above axiom \eqref{W5} is the most complicated, it may
be helpful to note that it may be seen as an instance of a pseudomonad
transformation axiom. This is explained by Gambino and Lobbia \cite[Remark 4.2]{GambinoDL}
who denote the axiom $\left(\textnormal{C5}\right)$.
\end{rem}

For convenience and easy reference, we also list the five redundant
coherence conditions of a pseudodistributive law. 
\begin{thm}
Given a pseudodistributive law $\left(\lambda,\omega_{1},\omega_{2},\omega_{3},\omega_{4}\right)\colon TP\to PT$
in pseudomonoidal form, the following five conditions are derivable.
\begin{equation}
\tag{{W6}}\xymatrix@=1.5em{TP\ar[rr]^{uTP}\ar[d]_{\lambda}\ar@/^{1.7pc}/[rrr]^{\textnormal{id}} &  & T^{2}P\ar[r]^{mP}\ar[d]_{T\lambda} & TP\ar[dd]^{\lambda}\\
PT\ar[rr]^{uPT}\ar[d]_{\textnormal{id}}\urtwocell[0.55]{drr}{\omega_{1}T} &  & TPT\urtwocell[0.7]{r}{\omega_{3}}\ar[d]_{\lambda T} & \; & \ar@{}[]|-{=} & \textnormal{id}_{\lambda}\\
PT\ar[rr]_{PuT}\ar@/_{1.7pc}/[rrr]_{\textnormal{id}} & \; & PT^{2}\ar[r]_{Pm} & PT
}
\label{W6}
\end{equation}
\begin{equation}
\tag{{W7}}\xymatrix@=1.5em{TP\ar[rr]^{TP\eta}\ar[d]_{\lambda}\ar@/^{1.7pc}/[rrr]^{\textnormal{id}} &  & TP^{2}\ar[r]^{T\mu}\ar[d]_{\lambda P} & TP\ar[dd]^{\lambda}\\
PT\ar[rr]^{PT\eta}\ar[d]_{\textnormal{id}}\dltwocell[0.45]{drr}{P\omega_{2}} & \; & PTP\ar[d]_{P\lambda}\dltwocell[0.5]{r}{\omega_{4}} & \; & \ar@{}[]|-{=} & \textnormal{id}_{\lambda}\\
PT\ar[rr]_{P\eta T}\ar@/_{1.7pc}/[rrr]_{\textnormal{id}} &  & P^{2}T\ar[r]_{\mu T} & PT
}
\label{W7}
\end{equation}
\begin{equation}
\tag{{W8}}\xymatrix@=1.5em{ & T^{2}P\ar[r]^{T\lambda}\ar[rd]^{mP} & TPT\ar[r]^{\lambda T} & PT^{2}\ar[rd]^{Pm} &  &  & T^{2}P\ar[r]^{T\lambda} & TPT\ar[r]^{\lambda T} & PT^{2}\ar[rd]^{Pm}\\
T^{2}\ar[ur]^{T^{2}\eta}\ar[rd]_{m} &  & TP\ar[rr]^{\lambda}\dtwocell[0.45]{dr}{\omega_{2}}\dtwocell[0.45]{ru}{\omega_{3}} &  & PT\ar@{}[r]|-{=} & T^{2}\ar[ur]^{T^{2}\eta}\ar[rd]_{m}\ar@/_{0.7pc}/[urr]_{T\eta T}\ar@/_{1.5pc}/[rurr]_{\eta T^{2}} & \dtwocell[0.6]{u}{T\omega_{2}} & \dtwocell[0.6]{u}{\omega_{2}T} &  & PT\\
 & T\ar@/_{1pc}/[rrru]_{\eta T}\ar[ur]_{T\eta} & \; & \; &  &  & T\ar@/_{1pc}/[rrru]_{\eta T} & \; & \;
}
\label{W8}
\end{equation}
\begin{equation}
\tag{{W9}}\xymatrix@=1.5em{ & P\ar@/^{1pc}/[rrrd]^{Pu}\ar[rd]^{uP} & \; & \; &  &  & P\ar@/^{1pc}/[rrrd]^{Pu} & \; & \;\\
P^{2}\ar[ru]^{\mu}\ar[rd]_{uP^{2}} &  & TP\ar[rr]^{\lambda}\dtwocell[0.45]{dr}{\omega_{4}}\dtwocell[0.45]{ur}{\omega_{1}} &  & PT\ar@{}[r]|-{=} & P^{2}\ar[ru]^{\mu}\ar[rd]_{uP^{2}}\ar@/^{0.7pc}/[rrd]^{PuP}\ar@/^{1.5pc}/[rrrd]^{P^{2}u} & \dtwocell[0.6]{d}{\omega_{1}P} & \dtwocell[0.6]{d}{P\omega_{1}} &  & PT\\
 & TP^{2}\ar[r]_{\lambda P}\ar[ur]_{T\mu} & PTP\ar[r]_{P\lambda} & P^{2}T\ar[ur]_{\mu T} &  &  & TP^{2}\ar[r]_{\lambda P} & PTP\ar[r]_{P\lambda} & P^{2}T\ar[ur]_{\mu T}
}
\label{W9}
\end{equation}
\begin{equation}
\tag{{W10}}\xymatrix{ & P\ar[rd]^{uP}\ar@/^{0.8pc}/[rrd]^{Pu} & \; &  &  &  & P\ar[rd]^{Pu}\\
1\ar[rd]_{u}\ar[ru]^{\eta} &  & TP\ar[r]^{\lambda}\dtwocell[0.4]{d}{\omega_{2}}\dtwocell[0.4]{u}{\omega_{1}} & PT & \ar@{}[]|-{=} & 1\ar[rd]_{u}\ar[ur]^{\eta} &  & PT\\
 & T\ar[ru]_{T\eta}\ar@/_{0.8pc}/[rru]_{\eta T} & \; &  &  &  & T\ar[ru]_{\eta T}
}
\label{W10}
\end{equation}
\end{thm}

\begin{rem}
The redundancy of \eqref{W6} and \eqref{W7} is shown directly by
Marmolejo and Wood \cite{marm2008}. However, this result can be seen
more easily by noting that pseudomonad morphisms can be seen as instances
of pseudoalgebras (as is well known in one dimension \cite{marm2002,noiteration1d}),
and that one of the unitality axioms for a pseudoalgebra is redundant
\cite[Lemma 9.1]{marm1997}. Curiously the methods of this paper give
another proof of the redundancy, though this proof would be less strong
as it uses additional pseudodistributive law axioms.
\end{rem}

\begin{rem}
Note that this is the best that one might hope for, in that only one
compatibly axiom is needed relating the pseudomonad morphism and opmorphism
data. This is why the set of five axioms given here is expected to
be minimal. Stated more precisely, this becomes the following result.
\end{rem}

\begin{thm}
\label{onlyW5} A pseudodistributive law $\left(\lambda,\omega_{1},\omega_{2},\omega_{3},\omega_{4}\right)\colon TP\to PT$
is equivalently a pseudomonad morphism $\left(\lambda,\omega_{1},\omega_{3}\right)\colon T\to T$
along $P$, and a pseudomonad opmorphism $\left(\lambda,\omega_{2},\omega_{4}\right)\colon P\to P$
along $T$, such that $\omega_{3}$ and $\omega_{4}$ satisfy axiom
(W5).
\end{thm}

\subsection{Decagon definition of pseudodistributive laws}

The following is the definition of pseudodistributive law one finds
after working out the conditions on a pseudonatural transformation
$\lambda\colon TP\to PT$ needed for extending a pseudomonad $\left(T,u,m\right)$
to the Kleisli bicategory of a pseudomonad $\left(P,\eta,\mu\right)$
in pseudoextensive form. In practice one would likely not use this
definition, but it will be needed for the later proofs and explanation
of redundant coherence axioms.
\begin{defn}
A \emph{pseudodistributive law (in Kleisli-decagon form)} of pseudomonads
 $\left(T,u,m\right)$ over $\left(P,\eta,\mu\right)$ is a pseudonatural
transformation $\lambda\colon TP\to PT$ and three invertible modifications
comprising the two triangles
\[
\xymatrix@=1.5em{TP\ar[rr]^{\lambda} & \; & PT &  &  & TP\ar[rr]^{\lambda} & \; & PT\\
 & P\ar[ru]_{Pu}\ar[lu]^{uP}\ultwocell[0.7]{u}{\omega_{1}} &  &  &  &  & T\ar[ru]_{\eta T}\ar[lu]^{T\eta}\drtwocell[0.6]{u}{\omega_{2}}
}
\]
and the decagon
\[
\xymatrix@=1em{TPTPT\ar[rr]^{TP\lambda T}\ar[d]_{\lambda TPT} &  & \dtwocell[0.5]{rrdd}{\Omega}TP^{2}T^{2}\ar[rr]^{TP^{2}m} &  & TP^{2}T\ar[rr]^{T\mu T} &  & TPT\ar[d]^{\lambda T}\\
PT^{2}PT\ar[d]_{PmPT} &  &  &  &  &  & PT^{2}\ar[d]^{Pm}\\
PTPT\ar[rr]_{P\lambda T} &  & P^{2}T^{2}\ar[rr]_{P^{2}m} &  & P^{2}T\ar[rr]_{\mu T} &  & PT
}
\]
satisfying the triangle equation;
\begin{equation}
\resizebox{0.85\textwidth}{!}{\tag{{D1}}\xymatrix{ &  & TPT\ar[rd]^{T\eta PT}\ar@/^{1pc}/[rrd]^{\textnormal{id}}\\
TPT\ar[d]|-{TuPT}\ar@/^{1pc}/[rru]^{\textnormal{id}}\ar@/_{2pc}/[ddd]_{\textnormal{id}}\ar[r]^{TPuT}\dtwocell[0.25]{dr}{T\omega_{1}T} & TPT^{2}\ar[r]^{T\eta PT^{2}}\ar[ru]^{TPm} & TP^{2}T^{2}\ar[r]^{TP^{2}m} & TP^{2}T\ar[r]^{T\mu T} & TPT\ar[dd]^{\lambda T}\\
T^{2}PT\myar{T\eta TPT}{r}\ar[d]|-{\textnormal{id}}\dtwocell[0.4]{dr}{\omega_{2}TPT}\ar@/_{0pc}/[ur]_{T\lambda T} & TPTPT\ar[d]^{\lambda TPT}\ar@/_{0pc}/[ur]_{TP\lambda T} & \dtwocell[0.45]{rdd}{\Omega}\\
T^{2}PT\ar@/_{0pc}/[r]_{\eta T^{2}PT}\ar[d]|-{mPT} & PT^{2}PT\ar[d]^{PmPT} &  &  & PT^{2}\ar[d]^{Pm}\ar@{}[r]|-{=} & \textnormal{id}_{Pm\cdot\lambda T}\\
TPT\ar[r]_{\eta TPT}\ar@/_{0.5pc}/[rd]_{\lambda T} & PTPT\ar[r]_{P\lambda T} & P^{2}T^{2}\ar[r]_{P^{2}m} & P^{2}T\ar[r]_{\mu T} & PT\\
 & PT^{2}\ar[ru]_{\eta PT^{2}}\ar[r]_{Pm} & PT\ar[ur]_{\eta PT}\ar@/_{1pc}/[rru]_{\textnormal{id}}
}
}\label{D1}
\end{equation}
the pentagon equation;
\begin{equation}
\resizebox{0.85\textwidth}{!}{\tag{{D2}}\xymatrix{TPTPTPT\ar[r]^{\underset{\;}{TPTP\lambda T}}\ar[d]_{\lambda TPTPT} & TPTP^{2}T^{2}\ar[d]|-{\lambda TP^{2}T^{2}}\ar[r]^{\underset{\;}{TPTP^{2}m}} & TPTP^{2}T\ar[d]|-{\lambda TP^{2}T}\ar[r]^{\underset{\;}{TPT\mu T}} & TPTPT\ar[d]^{\lambda TPT}\ar[rdr]^{\underset{\;}{TP\lambda T}}\\
PT^{2}PTPT\ar[d]_{PmPTPT}\ar[r]^{\underset{\;}{P^{2}TP\lambda T}} & PT^{2}P^{2}T^{2}\ar[d]^{PmP^{2}T^{2}}\ar[r]^{\underset{\;}{PT^{2}P^{2}m}} & PT^{2}P^{2}T\ar[d]^{PmP^{2}T}\ar[r]^{\underset{\;}{PT^{2}\mu T}} & PT^{2}PT\ar[d]^{PmPT} &  & TP^{2}T^{2}\ar[d]^{TP^{2}m}\\
PTPTPT\ar[r]^{PTP\lambda T}\ar[d]_{P\lambda TPT}\dtwocell[0.5]{ddrrr}{P\Omega} & PTP^{2}T^{2}\ar[r]^{PTP^{2}m} & PTP^{2}T\ar[r]^{PT\mu T} & PTPT\ar[d]^{P\lambda T} &  & TP^{2}T\ar[d]^{T\mu T}\\
P^{2}T^{2}PT\ar[d]_{P^{2}mPT} &  &  & P^{2}T^{2}\ar[d]^{P^{2}m}\dtwocell[0.5]{rr}{\Omega} &  & TPT\ar[d]^{\lambda T}\\
P^{2}TPT\ar[r]_{P^{2}\lambda T}\ar[rd]_{\mu TPT} & P^{3}T^{2}\ar[r]_{P^{3}m}\ar[rd]_{\mu PT^{2}} & P^{3}T\ar[r]_{P\mu T}\ar[rd]_{\mu PT} & P^{2}T\ar[rdr]_{\mu T} &  & PT^{2}\ar[d]^{Pm}\\
 & PTPT\ar[r]_{P\lambda T} & P^{2}T^{2}\ar[r]_{P^{2}m} & P^{2}T\ar[rr]_{\mu T} &  & PT
}
}\label{D2a}
\end{equation}
is equal to
\[
\resizebox{0.85\textwidth}{!}{\xymatrix{TPTPTPT\ar[rd]^{TP\lambda TPT}\ar[r]^{\underset{\;}{TPTP\lambda T}}\ar[d]_{\lambda TPTPT} & TPTP^{2}T^{2}\ar[r]^{\underset{\;}{TPTP^{2}m}} & TPTP^{2}T\ar[r]^{\underset{\;}{TPT\mu T}} & TPTPT\ar[dr]^{TP\lambda T}\\
PT^{2}PTPT\ar[d]_{PmPTPT} & TP^{2}T^{2}PT\ar[rd]^{TP^{2}mPT}\dtwocell[0.5]{rrr}{TP\Omega}\dtwocell[0.5]{ddd}{\Omega PT} &  &  & TP^{2}T^{2}\ar[rd]^{TP^{2}m}\\
PTPTPT\ar[d]_{P\lambda TPT} &  & TP^{2}TPT\ar[r]^{TP^{2}\lambda T}\ar[d]_{T\mu TPT} & TP^{3}T^{2}\ar[r]^{TP^{3}m}\ar[d]_{T\mu PT^{2}} & TP^{3}T\ar[r]^{TP\mu T}\ar[d]^{T\mu PT} & TP^{2}T\ar[d]^{T\mu T}\\
P^{2}T^{2}PT\ar[rd]_{P^{2}mPT} &  & TPTPT\ar[r]_{TP\lambda T}\ar[d]_{\lambda TPT}\dtwocell[0.5]{rrrdd}{\Omega} & TP^{2}T^{2}\ar[r]_{TP^{2}m} & TP^{2}T\ar[r]_{T\mu T} & TPT\ar[d]^{\lambda T}\\
 & P^{2}TPT\ar[dr]_{\mu TPT} & PT^{2}PT\ar[d]_{PmPT} &  &  & PT^{2}\ar[d]^{Pm}\\
 &  & PTPT\ar[r]_{P\lambda T} & P^{2}T^{2}\ar[r]_{P^{2}m} & P^{2}T\ar[r]_{\mu T} & PT
}
}
\]
and compatibility of $\Omega$ with $\mathscr{C}\to\mathscr{C}_{P}$;
\begin{equation}
\tag{{DC2}}\resizebox{0.85\textwidth}{!}{\xymatrix@=1.5em{TPT^{2}\ar[rr]^{TPT\eta T}\ar[d]^{\lambda T^{2}} &  & TPTPT\ar[rr]^{TP\lambda T}\ar[d]^{\lambda TPT} &  & \dtwocell[0.5]{rrdd}{\Omega}TP^{2}T^{2}\ar[rr]^{TP^{2}m} &  & TP^{2}T\ar[rr]^{T\mu T} &  & TPT\ar[d]^{\lambda T}\\
PT^{3}\ar[rr]^{PT^{2}\eta T}\ar[d]^{PmT} &  & PT^{2}PT\ar[d]^{PmPT} &  &  &  &  &  & PT^{2}\ar[d]^{Pm}\\
PT^{2}\ar[rr]_{PT\eta T}\ar@/_{1.5pc}/[rrrr]|-{P\eta T^{2}}\ar@/_{0.3pc}/[rrd]_{Pm} &  & PTPT\ar[rr]_{P\lambda T}\dtwocell[0.4]{d}{P\omega_{2}T} &  & P^{2}T^{2}\ar[rr]^{P^{2}m} &  & P^{2}T\ar[rr]^{\mu T} &  & PT\\
 &  & PT\ar[urrrr]_{P\eta T}\ar@/_{0.5pc}/[urrrrrr]_{\textnormal{id}}
}
}\label{DC2}
\end{equation}
is equal to
\[
\resizebox{0.85\textwidth}{!}{\xymatrix@=1.5em{TPT^{2}\ar[rr]^{TPT\eta T}\ar@/_{1.5pc}/[rrrr]|-{TP\eta T^{2}}\ar@/_{1.7pc}/[rrrrrdr]_{\textnormal{id}} &  & TPTPT\ar[rr]^{TP\lambda T}\dtwocell[0.4]{d}{TP\omega_{2}T} &  & TP^{2}T^{2}\ar[rr]^{TP^{2}m}\ar[rrd]_{T\mu T^{2}} &  & TP^{2}T\ar[rrrr]^{T\mu T} &  &  &  & TPT\ar[d]^{\lambda T}\\
 &  & \; &  &  &  & TPT^{2}\ar[rrrru]^{TPm}\ar[rr]_{\lambda T^{2}} &  & PT^{3}\ar[rr]_{PTm}\ar[d]_{PmT} &  & PT^{2}\ar[d]^{Pm}\\
 &  &  &  &  &  &  &  & PT^{2}\ar[rr]_{Pm} &  & PT\\
\\
}
}
\]
\end{defn}

\begin{rem}
When the data for enforcing compatibility is not part of the data
for forming the extension, the compatibility axioms are necessary.
This is why all of the later presentations will require both compatibility
axioms. 

Whilst this suggests axiom \eqref{DC2} might be redundant (because
the compatibility data $\omega_{2}$ is part of the extension data)\footnote{Another reason to expect \eqref{DC2} to be redundant is that the
other compatibility axiom is.}, the fact it does not involve the unit $u$ (whereas \eqref{D1}
does) suggests it may be necessary. The issue is that MacLane and
Par\'e's coherence theorem applies to $\omega_{1}$, $\Omega$ and
\[
\xymatrix@=1.5em{ &  &  &  & TPT\ar[rrd]^{\lambda T}\dtwocell[0.55]{d}{\omega_{2}T}\\
T\ar[rr]_{Tu} &  & T^{2}\ar[urr]^{T\eta T}\ar[rrrr]_{\eta T^{2}} & \; & \; &  & PT^{2}\ar[rr]_{Pm} &  & PT
}
\]
instead of $\omega_{2}$ (even though this simplifies to $\omega_{2}$
when composed with appropriate pseudomonad data)\footnote{Note this data directly appears in axiom \eqref{D1}}.
Hence the coherence theorem only applies to pasting diagrams constructed
in a certain way. This problem creates a persistent issue where the
redundancy of certain coherence conditions only becomes apparent when
restricted by a unit map $u$.
\end{rem}

\subsection{Pseudoalgebra definition of pseudodistributive laws}

The following is intended to provide a definition of pseudodistributive
laws involving a pseudoalgebra structure map $\alpha$ (with the pseudoalgebra
data and axioms being derivable from the law). This is in the spirit
of Marmolejo, Rosebrugh and Wood \cite[Prop. 3.5]{marm2002} who in
one dimension considered taking algebra structure maps as the data
for a distributive law, though directly assuming the algebra axioms.

We point out that as the usual monoidal definition of monads is natural,
defining a distributive law in monoidal form requires no extra ``change
of variables'' axioms. If the reader agrees the no-iteration definition
of monad is also natural, then the no-iteration (and by extension
pseudoalgebra) versions of a distributive law should also avoid any
such axioms. This is the conceptual reason as to why a simplification
of \cite[Prop. 3.5]{marm2002} is possible, reducing five axioms to
three in dimension one.\footnote{In the skew setting the monoidal and no-iteration definitions of skew
monad are distinct \cite{mndwarp} and so the reader is forced to
decide which one is natural. The same issue applies to distributive
laws in the skew setting (as the bijectivity of the change of variables
$\lambda\mapsto\alpha$ fails), forcing the reader to decide which
formulation is the natural one.}
\begin{defn}
A \emph{pseudodistributive law (in pseudoalgebra form)} of pseudomonads
 $\left(T,u,m\right)$ over $\left(P,\eta,\mu\right)$ is a pseudonatural
transformation $\alpha\colon TPT\to PT$ and three invertible modifications
\[
\xymatrix@=1em{TPT\myar{\alpha}{r}\ultwocell[0.45]{dr}{\psi} & PT & T^{2}\myar{T\eta T}{rr}\ar[dd]_{m}\dltwocell[0.5]{drdr}{\xi} &  & TPT\ar[dd]^{\alpha} & TPTPT\ar[rr]^{TP\alpha}\ar[dd]_{\alpha PT}\dtwocell[0.5]{drdrrr}{\Psi} &  & TP^{2}T\ar[rr]^{T\mu T} &  & TPT\ar[dd]^{\alpha}\\
 & \;\\
PT\ar[uu]^{uPT}\ar[ruu]_{\textnormal{id}} &  & T\myard{\eta T}{rr} &  & PT & PTPT\ar[rr]_{P\alpha} &  & P^{2}T\ar[rr]_{\mu T} &  & PT
}
\]
satisfying the triangle equation;
\begin{equation}
\resizebox{0.95\textwidth}{!}{\tag{{M1}}\xymatrix{TPT\ar[d]^{TuPT}\ar@/_{1.5pc}/[dd]_{\textnormal{id}}\ar@/^{2.5pc}/[rrrrrrr]^{\textnormal{id}}\ar@/^{0pc}/[rr]^{\textnormal{id}}\dtwocell[0.35]{rrd}{T\psi} &  & TPT\ar[rrr]^{T\eta PT} &  &  & TP^{2}T\ar[rr]^{T\mu T}\dtwocell[0.5]{dd}{\Psi} &  & TPT\ar[dd]^{\alpha}\\
T^{2}PT\ar[rrr]^{T\eta TPT}\ar[d]^{mPT}\dtwocell[0.45]{rrrd}{\xi PT}\ar@/_{0.5pc}/[rur]_{T\alpha} & \; & \; & TPTPT\ar[d]_{\alpha PT}\ar@/_{0.5pc}/[rur]_{TP\alpha} & \; &  &  &  & \ar@{}[]|-{=} & \textnormal{id}_{\alpha}\\
TPT\ar[rrr]_{\eta TPT}\ar@/_{0.6pc}/[rrrd]_{\alpha} &  &  & PTPT\ar[rr]_{P\alpha} &  & P^{2}T\ar[rr]_{\mu T} &  & PT\\
 &  &  & PT\ar[rru]_{\eta PT}\ar@/_{0.8pc}/[rrrru]_{\textnormal{id}}
}
}\label{M1}
\end{equation}
the pentagon equation;
\begin{equation}
\tag{{M2}}\resizebox{0.95\textwidth}{!}{\xymatrix@=1.5em{TPTPTPT\ar[dd]_{\alpha PTPT}\ar[rr]^{TPTP\alpha} &  & TPTP^{2}T\ar[rr]^{TPT\mu T}\ar[dd]_{\alpha P^{2}T} &  & TPTPT\ar[rrd]^{TP\alpha}\ar[dd]_{\alpha PT}\\
 &  &  &  &  &  & TP^{2}T\ar[rrd]^{T\mu T}\dtwocell[0.5]{dd}{\Psi}\\
PTPTPT\ar[rrd]_{P\alpha PT}\ar[rr]^{PTP\alpha} &  & PTP^{2}T\ar[rr]^{PT\mu T}\dtwocell[0.5]{rrd}{P\Psi} &  & PTPT\ar[rrd]^{P\alpha} &  &  &  & TPT\ar[dd]^{\alpha}\\
 &  & P^{2}TPT\ar[rrd]_{\mu TPT}\ar[rr]_{P^{2}\alpha} &  & P^{3}T\ar[rr]_{P\mu T}\ar[rrd]_{\mu PT} &  & P^{2}T\ar[rrd]^{\mu T}\\
 &  &  &  & PTPT\ar[rr]_{P\alpha} &  & P^{2}T\ar[rr]_{\mu T} &  & PT
}
}\label{M2}
\end{equation}
is equal to
\[
\resizebox{0.95\textwidth}{!}{\xymatrix@=1.5em{TPTPTPT\ar[dd]_{\alpha PTPT}\ar[rr]^{TPTP\alpha}\ar[rrd]_{TP\alpha PT} &  & TPTP^{2}T\ar[rr]^{TPT\mu T} &  & TPTPT\ar[rrd]^{TP\alpha}\dtwocell[0.5]{ld}{TP\Psi}\\
 &  & TP^{2}TPT\ar[rrd]_{T\mu TPT}\dtwocell[0.5]{dd}{\Psi PT}\ar[rr]^{TP^{2}\alpha} & \; & TP^{3}T\ar[rr]^{TP\mu T}\ar[rrd]^{T\mu PT} &  & TP^{2}T\ar[rrd]^{T\mu T}\\
PTPTPT\ar[rrd]_{P\alpha PT} &  &  &  & TPTPT\ar[dd]^{\alpha PT}\ar[rr]_{TP\alpha} &  & TP^{2}T\ar[rr]_{T\mu T}\dtwocell[0.5]{dd}{\Psi} &  & TPT\ar[dd]^{\alpha}\\
 &  & P^{2}TPT\ar[rrd]_{\mu TPT}\\
 &  &  &  & PTPT\ar[rr]_{P\alpha} &  & P^{2}T\ar[rr]_{\mu T} &  & PT
}
}
\]
compatibility of $\psi$ with $\mathscr{C}\to\mathscr{C}_{P}$;
\begin{equation}
\tag{{MC1}}\resizebox{0.5\textwidth}{!}{\xymatrix@=1.5em{ &  & T\ar[dd]_{uT}\ar[rr]^{\eta T}\ar@/_{2pc}/[dddd]_{\textnormal{id}} &  & PT\ar[dd]_{uPT}\ar@/^{3pc}/[dddd]^{\textnormal{id}}\\
\\
 &  & T^{2}\myar{T\eta T}{rr}\ar[dd]_{m}\dltwocell[0.5]{drdr}{\xi} &  & TPT\ar[dd]^{\alpha} & \dltwocell[0.4]{l}{\psi} & = & \textnormal{id}_{\eta T}\\
\\
 &  & T\myard{\eta T}{rr} &  & PT
}
}\label{MC1}
\end{equation}
and compatibility of $\Psi$ with $\mathscr{C}\to\mathscr{C}_{P}$;
\begin{equation}
\tag{{MC2}}\resizebox{0.8\textwidth}{!}{\xymatrix@=1.5em{ & T^{3}\ar[rr]^{T\eta T^{2}}\ar[dd]_{mT} &  & TPT^{2}\ar[dd]_{\alpha T}\ar[rr]^{TPT\eta T} &  & TPTPT\ar[rr]^{TP\alpha}\ar[dd]_{\alpha PT}\dtwocell[0.5]{drdrrr}{\Psi} &  & TP^{2}T\ar[rr]^{T\mu T} &  & TPT\ar[dd]^{\alpha}\\
 &  & \dtwocell[0]{d}{\xi T}\\
 & T^{2}\ar[rr]_{\eta T^{2}}\ar[rrd]_{m} &  & PT^{2}\ar[rr]_{PT\eta T}\ar[rrd]_{Pm} &  & PTPT\ar[rr]_{P\alpha}\dtwocell[0.5]{d}{P\xi} &  & P^{2}T\ar[rr]_{\mu T} &  & PT\\
 &  &  & T\ar[rr]_{\eta T} &  & PT\ar[rru]_{P\eta T}\ar@/_{0.5pc}/[rrurr]_{\textnormal{id}}
}
}\label{MC2}
\end{equation}
is equal to
\[
\resizebox{0.8\textwidth}{!}{\xymatrix@=1.5em{ & T^{3}\ar[rr]^{T\eta T^{2}}\ar[rrd]_{Tm}\ar[dd]_{mT} &  & TPT^{2}\ar[rr]^{TPT\eta T}\ar[rrd]_{TPm} &  & TPTPT\ar[rr]^{TP\alpha}\dtwocell[0.5]{d}{TP\xi} &  & TP^{2}T\ar[rr]^{T\mu T} &  & TPT\ar[dd]^{\alpha}\\
 &  &  & T^{2}\ar[rr]_{T\eta T}\ar[d]_{m} &  & TPT\ar[rru]_{TP\eta T}\ar@/_{0.5pc}/[rrurr]_{\textnormal{id}}\dtwocell[0.5]{dr}{\xi}\\
 & T^{2}\ar[rr]_{m} &  & T\ar[rrrrrr]_{\eta T} &  &  & \; &  &  & PT\\
\\
}
}
\]
\end{defn}

\begin{rem}
Most of the technicalities in changing to this formulation arise from
replacing the data $\omega_{2}$ with the new data $\xi$, which now
serves two separate roles. Firstly, the data $\xi$ is what is required
to provide compatibility with the Kleisli functor $\mathscr{C}\to\mathscr{C}_{P}$.
Secondly, it yields the pasting

\[
\xymatrix@=1em{T\ar[rr]^{Tu}\ar@/_{0.5pc}/[rrdd]_{\textnormal{id}} &  & T^{2}\myar{T\eta T}{rr}\ar[dd]_{m}\dltwocell[0.5]{drdr}{\xi} &  & TPT\ar[dd]^{\alpha}\\
\\
 &  & T\myard{\eta T}{rr} &  & PT
}
\]
which is the data required to construct an extension (without enforcing
compatibility). In this way one may view $\xi$ as two pieces of coherence
data combined into one. In fact, asking that these two pieces of data
are related this way is the condition for compatibility of $\xi$
with $\mathscr{C}\to\mathscr{C}_{P}$.
\end{rem}

\begin{rem}
The restriction of \eqref{MC1} along the unit $u\colon1\to T$ and
the restriction of \eqref{MC2} along $T^{2}u\colon T^{2}\to T^{3}$
are both redundant in that they follow from the pair of axioms \eqref{M1}
and \eqref{M2}\footnote{The former is straightforward to check and the latter is a pseudoalgebra
version of \eqref{D4}.}.
\end{rem}

\subsection{No-iteration definition of pseudodistributive laws}

The following defines a pseudodistributive law in terms of pseudo-pasting
operators $\left(-\right)_{X,Y}^{\lambda}$, which in our case arise
as families of functors $\mathscr{C}\left(X,PTY\right)\to\mathscr{C}\left(TX,PTY\right)$
induced by pasting with a 2-cell $\alpha\colon TPT\Rightarrow PT$
as in the diagram
\[
\xymatrix@=1em{ &  &  &  & \mathscr{C}\ar[rrdd]^{T}\\
 &  & \dtwocell[0.2]{ld}{f} &  & \dtwocell[0.2]{d}{\alpha}\\
\mathscr{B}\ar@/^{0.7pc}/[ururrr]^{X}\ar[rr]_{Y} & \; & \mathscr{C}\ar[urur]^{PT}\ar[rrrr]_{PT} &  & \; &  & \mathscr{C}.
}
\]
This definition (which generalizes \cite[Definition 2.1]{noiteration1d}
to two dimensions) should be useful in that it more readily generalizes
to the relative case \cite{relative} (in which one may no longer
have such a 2-cell $\alpha$).
\begin{defn}
\label{pasting} A \emph{pseudo-pasting operator} in a tricategory
$\mathscr{K}$
\[
\left(-\right)^{\#}\colon\mathscr{K}\left(C,D\right)\left(-,s-\right)\to\mathscr{K}\left(C,E\right)\left(t-,u-\right)
\]
is a family of functors indexed by $f\colon C\to D$ and $g\colon C\to\textnormal{dom}s=\textnormal{dom}u$
\[
\left(-\right)_{f,g}^{\#}\colon\mathscr{K}\left(C,D\right)\left(f,sg\right)\to\mathscr{K}\left(C,E\right)\left(tf,ug\right)
\]
equipped with:
\begin{enumerate}
\item (whiskering data) for all $\vartheta\colon f\Rightarrow sg\colon C\to D$
and $h\colon A\to C$ equalities $\vartheta^{\#}h=\left(\vartheta h\right)^{\#}$\footnote{We use an equality here since it holds strictly in the cases of interest.
In particular for a pseudonatural transformation $\alpha$ and whiskering
by a pseudofunctor $F$ we have $\left(\alpha F\right)_{X}:=\alpha_{FX}$.};
\item (blistering data) for all $\vartheta\colon f\Rightarrow sg\colon C\to D$
and $\xi\colon p\Rightarrow f\colon C\to D$ a 3-isomorphism $\textnormal{bl}_{\#}\left(\vartheta,\xi\right)\colon\vartheta^{\#}\left(t\xi\right)\cong\left(\vartheta\xi\right)^{\#}$;
\end{enumerate}
such that using $\circ$ for vertical composition and $\cdot$ for
horizontal whiskering:
\begin{enumerate}
\item (respecting naturality) all $\nu\colon\vartheta\Rrightarrow\vartheta'\colon f\Rightarrow sg\colon C\to D$
render commutative
\[
\xymatrix@=1em{\vartheta^{\#}\circ\left(t\cdot\xi\right)\ar[d]_{\nu^{\#}\circ\left(t\cdot\xi\right)}\ar[rrrr]^{\textnormal{bl}_{\#}\left(\vartheta,\xi\right)} &  &  &  & \left(\vartheta\circ\xi\right)^{\#}\ar[d]^{\left(\nu\circ\xi\right)^{\#}}\\
\vartheta'^{\#}\circ\left(t\cdot\xi\right)\myard{\textnormal{bl}_{\#}\left(\vartheta',\xi\right)}{rrrr} &  &  &  & \left(\vartheta'\circ\xi\right)^{\#}
}
\]
\item (respecting whiskering) all such $\vartheta$, $\xi$ and $h$ render
commutative
\[
\xymatrix@=1em{\left(\vartheta^{\#}\circ\left(t\cdot\xi\right)\right)\cdot h\ar[rrrrrr]^{\textnormal{bl}_{\#}\left(\vartheta,\xi\right)\cdot h}\ar[d]_{\textnormal{mid four int }} &  &  &  &  &  & \left(\vartheta\circ\xi\right)^{\#}\cdot h\ar@{=}[dd]\\
\left(\vartheta^{\#}\cdot h\right)\circ\left(t\cdot\left(\xi\cdot h\right)\right)\ar@{=}[d]\\
\left(\vartheta\cdot h\right)^{\#}\circ\left(t\cdot\left(\xi\cdot h\right)\right)\myard{\textnormal{bl}_{\#}\left(\vartheta\cdot h,\xi\cdot h\right)}{rrr} &  &  & \left(\vartheta\circ\left(\xi\cdot h\right)\right)^{\#}\ar[rrr]_{\textnormal{mid four int}^{\#}} &  &  & \left(\left(\vartheta\circ\xi\right)\cdot h\right)^{\#}
}
\]
\item (respecting blistering) all $\psi\colon q\Rightarrow p\colon C\to D$
render commutative
\[
\xymatrix@=1em{\vartheta^{\#}\circ\left(t\cdot\left(\xi\circ\psi\right)\right)\ar[rrrrrr]^{\textnormal{bl}_{\#}\left(\vartheta,\xi\circ\psi\right)}\ar[d]_{\textnormal{mid four int }} &  &  &  &  &  & \left(\vartheta\circ\left(\xi\circ\psi\right)\right)^{\#}\ar[d]^{\textnormal{ assoc}_{\;}^{\#}}\\
\vartheta^{\#}\circ\left(t\cdot\xi\right)\circ\left(t\cdot\psi\right)\ar[rrr]\myard{\textnormal{bl}_{\#}\left(\vartheta,\xi\right)\circ\left(t\cdot\psi\right)}{rrr} &  &  & \left(\vartheta\circ\xi\right)^{\#}\circ\left(t\cdot\psi\right)\ar[rrr]\myard{\textnormal{bl}_{\#}\left(\vartheta\circ\xi,\psi\right)}{rrr} &  &  & \left(\left(\vartheta\circ\xi\right)\circ\psi\right)^{\#}
}
\]
\end{enumerate}
\end{defn}

\begin{rem}
When $\mathscr{K}$ is a 2-category each $\textnormal{bl}_{\#}\left(\vartheta,\xi\right)$
is an equality and we recover the usual one dimensional version of
pasting operators.
\end{rem}

\begin{rem}
In our case of interest we will take $\mathscr{K}$ to be the tricategory
of 2-categories, pseudofunctors, pseudonatural transformations and
modifications, $D=E=\mathscr{C}$, $t=T$ and $s=u=PT$. To replace
$f$, we denote by $X\in\mathscr{C}$ a generalized element $X\colon\mathscr{I}\to\mathscr{C}$
of $\mathscr{C}$ so that $C=\mathscr{I}$ and a pseudo-pasting operator
consists of functors $\mathscr{\mathscr{K}}\left(\mathscr{I},\mathscr{C}\right)\left(X,PTY\right)\to\mathscr{K}\left(\mathscr{I},\mathscr{C}\right)\left(TX,PTY\right)$.
This may be further simplified since it is equivalent to a coherent
family in which each $\mathscr{I}$ is the terminal 2-category\footnote{We remind the reader that in the case of a no-iteration monad (or
an extension system) the naturality of the unit and multiplication
may be shown as a consequence of the definition. Indeed, a similar
calculation may be done for the distributive law version concerning
the resulting $\alpha\colon TPT\to PT$. Here $\left(PTf\right)^{\lambda}\cong\alpha_{Y}\cdot TPTf$
by blistering, and the isomorphism $\left(PTf\right)^{\lambda}\cong PTf\cdot\alpha_{X}$
is $\Psi^{\eta TY\cdot uY\cdot f,\textnormal{id}_{PTX}}$ simplified
using $\xi^{uY\cdot f}$. This means one may use global elements in
place of generalized elements.}.
\end{rem}

\begin{rem}
A basic example of a pseudo-pasting operator is $\left(-\right)^{u}$
induced by a pseudonatural transformation $u\colon\textnormal{1}_{\mathscr{C}}\to T$.
\end{rem}

We can now give the no-iteration definition of a pseudodistributive
law, which involves no iteration of the pseudomonads $T$ and $P$.
\begin{defn}
\label{psnoit} A \emph{pseudodistributive law (in no-iteration form)}
of pseudomonads  $\left(T,u,\left(-\right)^{T}\right)$ over $\left(P,\eta,\left(-\right)^{P}\right)$
on a 2-category $\mathscr{C}$ is a pseudo-pasting operator
\[
\left(-\right)_{X,Y}^{\lambda}\colon\mathscr{C}\left(X,PTY\right)\to\mathscr{C}\left(TX,PTY\right),\qquad X,Y\in\mathscr{C}
\]
along with for all $f\colon X\to PTY$, $k\colon X\to TY$, and $g\colon Y\to PTZ$
a family of invertible 2-cells
\[
\xymatrix@=1em{X\ar@/^{0pc}/[dd]_{uX}\ar@/^{1pc}/[rrdd]^{f} &  &  &  & TX\ar[dd]_{k^{T}}\ar@/^{1pc}/[rrdd]^{\left(\eta TY\cdot k\right)^{\lambda}} &  &  &  & TX\ar[dd]_{f^{\lambda}}\ar@/^{1pc}/[rrdd]^{\left(\left(g^{\lambda}\right)^{P}f\right)^{\lambda}}\\
\; & \dltwocell[0.1]{l}{\psi^{f}} &  &  & \; & \dltwocell[0.1]{l}{\xi^{k}} &  &  & \; & \dltwocell[0.15]{l}{\Psi^{g,f}}\\
TX\ar[rr]_{f^{\lambda}} & \; & PTY &  & TY\myard{\eta TY}{rr} &  & PTY &  & PTY\ar[rr]_{\left(g^{\lambda}\right)^{P}} &  & PTZ
}
\]
natural in $f,k$, $g$ and respecting whiskering and blistering\footnote{The whiskering axiom ensures $f^{\lambda}K=\left(fK\right)^{\lambda}$
and the blistering axiom ensures $g^{\lambda}Tf=\left(gf\right)^{\lambda}$.
That the family of $\psi^{f}$ respects these axioms ensures the following:
with $X,Y\colon\mathscr{B}\to\mathscr{C}$ and $K\colon\mathscr{A}\to\mathscr{B}$
we have $\psi^{fK}=\psi^{f}K$ (whiskering), and $\psi_{f}$ may be
composed by a naturality square of $u$ at $p\colon W\to X$ to give
$\psi_{fp}$ (blistering). Both families of $\xi^{k}$ and $\Psi^{g,f}$
satisfy similar conditions. Like naturality, we will avoid detailing
each whiskering and blistering condition as they are uninteresting.
If the reader requires them, they may be derived from the warping
formulation as instances of the axioms of morphisms of pseudo-pasting
operators.}, such that we have for all $g\colon Y\to PTZ$ the triangle equation;
\begin{equation}
\tag{{I1}}\xymatrix@=1em{TY\ar[rrrrrr]^{g^{\lambda}}\ar@{=}[dd]\ar@/_{2pc}/[dddddddd]_{\textnormal{id}} &  &  & \; &  &  & PTZ\ar@{=}[dd]\\
 &  & \dtwocell[0.6]{r}{\left(\psi^{g}\right)^{\lambda}} & \;\\
TY\ar@{=}[dd]\ar[rrrrrr]_{\left(g^{\lambda}uY\right)^{\lambda}} &  &  &  & \; &  & PTZ\ar@{=}[dd]\\
\\
TY\ar@{=}[dd]\ar[rrrrrr]^{\left(\left(g^{\lambda}\right)^{P}\cdot\eta TY\cdot uY\right)^{\lambda}} &  &  & \; &  &  & PTZ\ar@{=}[dd] & \ar@{}[]|-{=} & \textnormal{id}_{g^{\lambda}}\\
 &  & \; & \;\dtwocell[0.1]{u}{\Psi^{g,\eta TY\cdot uY}}\\
TY\ar[rrr]^{\left(\eta TY\cdot uY\right)^{\lambda}}\ar[dd]^{\left(uY\right)^{T}} &  &  & PTY\ar@/_{0pc}/[rrr]^{\left(g^{\lambda}\right)^{P}}\ar@{=}[dd] &  &  & PTZ\ar@{=}[dd]\\
 & \dtwocell[0.7]{r}{\xi^{uY}} & \; &  & \;\\
TY\ar[rrr]_{\eta TY}\ar@/_{2pc}/[rrrrrr]_{g^{\lambda}} &  &  & PTY\ar@/_{0pc}/[rrr]_{\left(g^{\lambda}\right)^{P}} & \; &  & PTZ
}
\label{I1}
\end{equation}
for all $f\colon X\to PTY$, $g\colon Y\to PTZ$ and $h\colon Z\to PTW$
the pentagon equation;
\begin{equation}
\tag{{I2}}\xymatrix@=1em{TX\ar[rrrrrr]^{\left(\left(\left(\left(h^{\lambda}\right)^{P}g\right)^{\lambda}\right)^{P}f\right)^{\lambda}}\ar@{=}[dd]\dtwocell[0.42]{rrdd}{\left(\Psi^{\left(h^{\lambda}\right)^{P}g,f}\right)^{P}} &  &  &  &  &  & PTW\ar@{=}[dd]\\
 & \; &  & \;\\
TX\ar@{=}[dd]\ar[rr]^{f^{\lambda}} &  & PTY\ar[rrrr]^{\left(\left(\left(h^{\lambda}\right)^{P}g\right)^{\lambda}\right)^{P}}\ar@{=}[dd] &  &  &  & PTW\ar@{=}[dd]\\
 &  &  &  & \dtwocell[0]{}{\left(\Psi^{h,g}\right)^{P}}\\
TX\ar@{=}[dd]\ar[rr]^{f^{\lambda}} &  & PTY\ar[rrrr]_{\left(\left(h^{\lambda}\right)^{P}g^{\lambda}\right)^{P}}\ar@{=}[dd] &  &  &  & PTW\ar@{=}[dd]\\
\\
TX\ar[rr]_{f^{\lambda}} &  & PTY\ar[rr]_{\left(g^{\lambda}\right)^{P}} &  & PTZ\ar[rr]_{\left(h^{\lambda}\right)^{P}} &  & PTW
}
\label{I2}
\end{equation}
is equal to
\[
\xymatrix@=1em{TX\ar[rrrrrr]^{\left(\left(\left(\left(h^{\lambda}\right)^{P}g\right)^{\lambda}\right)^{P}f\right)^{\lambda}}\ar@{=}[dd] &  &  &  &  &  & PTW\ar@{=}[dd]\\
 &  & \; & \dtwocell[0.5]{l}{\left(\left(\Psi^{h,g}\right)^{P}f\right)^{\lambda}}\\
TX\ar@{=}[dd]\ar[rrrrrr]|-{\left(\left(\left(h^{\lambda}\right)^{P}g^{\lambda}\right)^{P}f\right)^{\lambda}} &  &  &  &  &  & PTW\ar@{=}[dd]\\
\\
TX\ar[rrrrrr]|-{\left(\left(h^{\lambda}\right)^{P}\left(g^{\lambda}\right)^{P}f\right)^{\lambda}}\ar@{=}[dd] &  &  &  &  &  & PTW\ar@{=}[dd]\\
 &  &  & \; & \dtwocell[0.9]{l}{\underset{\;}{\Psi^{h,\left(g^{\lambda}\right)^{P}f}}}\\
TX\ar[rrrr]^{\left(\left(g^{\lambda}\right)^{P}f\right)}\ar@{=}[dd] &  &  &  & PTZ\ar[rr]^{\left(h^{\lambda}\right)^{P}}\ar@{=}[dd] &  & PTW\ar@{=}[dd]\\
 &  & \dtwocell[0.5]{r}{\Psi_{\;}^{g,f}} & \;\\
TX\ar[rr]_{f^{\lambda}} &  & PTY\ar[rr]_{\left(g^{\lambda}\right)^{P}} &  & PTZ\ar[rr]_{\left(h^{\lambda}\right)^{P}} &  & PTW
}
\]
for all $k\colon X\to TY$ compatibility of $\psi$ with $\mathscr{C}\to\mathscr{C}_{P}$;
\begin{equation}
\tag{{IC1}}\xymatrix@=1em{X\ar[dd]^{uX}\ar@/_{1pc}/[dddd]_{k}\ar[rrr]^{k} &  &  & TY\ar[dd]^{\eta TY}\\
 & \dtwocell[0.2]{r}{\psi^{\eta TY\cdot k}} & \;\\
TX\ar[rrr]^{\left(\eta TY\cdot k\right)^{\lambda}}\ar[dd]^{k^{T}} &  &  & PTY\ar@{=}[dd] & \ar@{}[]|-{=} & \textnormal{id}_{\eta TY\cdot k}\\
 & \dtwocell[0.7]{r}{\xi^{k}} & \;\\
TY\ar[rrr]_{\eta TY} &  &  & PTY
}
\label{IC1}
\end{equation}
and for all $f\colon X\to TY$ and $g\colon Y\to TZ$ compatibility
of $\Psi$ with $\mathscr{C}\to\mathscr{C}_{P}$;
\begin{equation}
\tag{{IC2}}\xymatrix@=1em{TX\ar@{=}[dd]\ar[rrrr]^{\left(\left(\eta Z\cdot g\right)^{\lambda}f\right)^{\lambda}} &  &  &  & PTZ\ar@{=}[dd]\\
\\
TX\ar[rrrr]^{\left(\left(\left(\eta Z\cdot g\right)^{\lambda}\right)^{P}\eta Y\cdot f\right)^{\lambda}}\ar@{=}[dd] &  &  &  & PTZ\ar@{=}[dd] &  & TX\ar@{=}[dd]\ar[rrrr]^{\left(\left(\eta Z\cdot g\right)^{\lambda}f\right)^{\lambda}} &  & \; &  & PTZ\ar@{=}[dd]\\
 & \; & \dtwocell[0.2]{l}{\Psi_{\;}^{\eta Z\cdot g,\eta Y\cdot f}} & \; &  &  &  & \dtwocell[0.1]{rrru}{\left(\xi^{g}f\right)^{\lambda}} & \;\\
TX\ar[rr]_{\left(\eta Y\cdot f\right)^{\lambda}}\ar[dd]_{f^{T}} &  & PTY\ar[rr]_{\left(\left(\eta Z\cdot g\right)^{\lambda}\right)^{P}}\ar@{=}[dd] &  & PTZ\ar@{=}[dd] & = & TX\ar[rrrr]^{\left(\eta Z\cdot g^{T}\cdot f\right)^{\lambda}}\ar[dd]_{g^{T}f^{T}} &  & \; &  & PTZ\ar@{=}[dd]\\
 & \dtwocell[0.2]{ld}{\xi^{f}} & \; & \dtwocell[0.2]{ld}{\left(\xi^{g}\right)^{P}} &  &  &  & \dtwocell[0.1]{rrrd}{\xi^{g^{T}f}} & \; & \;\\
TY\ar[rr]_{\eta TY}\ar@{=}[dd] &  & PTY\ar[rr]_{\left(\eta Z\cdot g^{T}\right)^{P}} &  & PTZ\ar@{=}[dd] &  & TZ\ar[rrrr]_{\eta TZ} &  &  &  & PTZ\\
\\
TY\ar[rrrr]_{\eta Z\cdot g^{T}} &  &  &  & PTZ
}
\label{IC2}
\end{equation}
\end{defn}

\begin{rem}
The data $\xi^{uY}$ suffices for constructing the extension (the
first two axioms) whereas the the more general data $\xi^{k}$ is
required for compatibility (the remaining two axioms).
\end{rem}

\subsection{Warping definition of pseudodistributive laws}

In \cite{mndwarp} Street and Lack showed that pseudomonads presented
in extensive form are in equivalence with warpings. These (pseudomonad)
warpings correspond to Kleisli bicategory structures, with the two
unitors and associator of the pseudomonad being used to construct
the two unitors and associator of the Kleisli bicategory. Conversely
the two unitors and associator of the pseudomonad can be recovered
from that of the Kleisli bicategory structure. It is by this correspondence,
that MacLane and Paré's coherence theorem \cite{MacPare} which applies
to the Kleisli bicategory, must apply to the pseudomonad data also.
This is how three of the no-iteration pseudomonad axioms were shown
to be redundant \cite{mndwarp}.

Here we extend this correspondence from pseudomonads to pseudodistributive
laws, using the fact that these laws correspond to extended pseudomonads
$\widetilde{T}$ on the Kleisli bicategory $\mathscr{C}_{P}$. The
technicality here is that pseudodistributive laws are more than just
a $\widetilde{T}$, as we also have compatibility data and axioms.
Hence the coherence theorem only applies to the Kleisli bicategory
of $\widetilde{T}$, constructed from the pseudodistributive law data
$\xi^{uX}$ and arbitrary $\psi$ and $\Psi$, but not the more general
compatibility data $\xi^{k}$, explaining the need for two additional
compatibility axioms.

The definition of warping used by Street and Lack \cite{mndwarp}
does not directly mention pasting operators due to simplifications
which are possible in that reduced case. In our more general setting
however we will use both pseudo-pasting operators and their morphisms,
as it is unclear if one can give a natural definition without them.
\begin{defn}
\label{pastingmor} Given two pseudo-pasting operators in a tricategory
$\mathscr{K}$
\[
\left(-\right)^{\#},\left(-\right)^{\#'}\colon\mathscr{K}\left(C,D\right)\left(-,s-\right)\to\mathscr{K}\left(C,E\right)\left(t-,u-\right)
\]
a\emph{ morphism of pseudo-pasting operators} \emph{$\left(-\right)^{\wp}\colon\left(-\right)^{\#}\rightsquigarrow\left(-\right)^{\#'}$}
is a family of natural transformations indexed by $f\colon C\to D$
and $g\colon C\to\textnormal{dom}s=\textnormal{dom}u$
\[
\left(-\right)_{f,g}^{\wp}\colon\left(-\right)_{f,g}^{\#}\rightsquigarrow\left(-\right)_{f,g}^{\#'}\colon\mathscr{K}\left(C,D\right)\left(f,sg\right)\to\mathscr{K}\left(C,E\right)\left(tf,ug\right)
\]
such that:
\begin{enumerate}
\item (whiskering) for all $\vartheta\colon f\Rightarrow sg\colon C\to D$
and $h\colon A\to C$ we have $\vartheta^{\wp}h=\left(\vartheta h\right)^{\wp}$;
\item (blistering) all $\vartheta\colon f\Rightarrow sg\colon C\to D$ and
$\xi\colon p\Rightarrow f\colon C\to D$ render commutative
\[
\xymatrix@=1em{\vartheta^{\#}\left(t\xi\right)\ar[d]_{\vartheta^{\wp}\left(t\xi\right)}\ar[rrrr]^{\textnormal{bl}_{\#}\left(\vartheta,\xi\right)} &  &  &  & \left(\vartheta\xi\right)^{\#}\ar[d]^{\left(\vartheta\xi\right)^{\wp}}\\
\vartheta^{\#'}\left(t\xi\right)\ar[rrrr]_{\textnormal{bl}_{\#'}\left(\vartheta,\xi\right)} &  &  &  & \left(\vartheta\xi\right)^{\#'}
}
\]
\end{enumerate}
\end{defn}

This formulation of pseudodistributive laws is designed to look as
similar as possible to that of warpings for pseudomonads, and thus
might be called a ``distributivity warping''. Whilst this warping
form may appear to be a complicated formulation, it provides the machinery
which simplifies the previous formulations and also gives the data
of the resulting compatible Kleisli bicategory structure. It may be
possible to more easily derive this formulation by iterating the notion
of warping (and enforcing compatibility).
\begin{rem}
The definition of pasting operators is not suggestive of closure under
composition, though those we are considering certainly are composable
as they correspond to 2-cells. This is why a number of composition
functors appear in the following definition.
\end{rem}

\begin{defn}
\label{pswarp} A \emph{pseudodistributive law (in warping form)}
of pseudomonads  $\left(T,u,\left(-\right)^{T}\right)$ over $\left(P,\eta,\left(-\right)^{P}\right)$
on a 2-category $\mathscr{C}$ consists of:
\begin{itemize}
\item a pseudo-pasting operator $\left(-\right)_{X,Y}^{\lambda}\colon\mathscr{C}\left(X,PTY\right)\to\mathscr{C}\left(TX,PTY\right)$;
\item morphisms of pseudo-pasting operators
\[
\xymatrix@=1em{ &  &  & \dtwocell[0.3]{dl}{\left(-\right)_{X,Y}^{\psi}}\\
\mathscr{C}\left(X,PTY\right)\myard{\left(-\right)_{X,Y}^{\lambda}\cdot\left(-\right)_{X}^{u}}{rrr}\ar@/^{2pc}/[rrrrr]^{\textnormal{id}} &  & \; & \mathscr{C}\left(TX,PTY\right)\mathscr{C}\left(X,TX\right)\myard{\circ}{rr} &  & \mathscr{C}\left(X,PTY\right)
}
\]
\[
\xymatrix@=1em{\mathscr{C}\left(X,TY\right)\myar{\left(-\right)_{X,TY}^{\eta}}{rrr}\myard{\left(-\right)_{TY}^{\eta}\cdot\left(-\right)_{X,Y}^{T}\quad\quad}{rrrd} &  & \; & \mathscr{C}\left(X,PTY\right)\myar{\left(-\right)_{X,Y}^{\lambda}}{rr}\dtwocell[0.5]{d}{\left(-\right)_{X,Y}^{\xi}} &  & \mathscr{C}\left(TX,PTY\right)\\
 &  &  & \mathscr{C}\left(TY,PTY\right)\mathscr{C}\left(TX,TY\right)\myard{\circ}{rur}
}
\]
\[
\xymatrix@=1em{\mathscr{C}\left(Y,PTZ\right)\mathscr{C}\left(X,PTY\right)\myar{\left(\left(-\right)_{Y,Z}^{\lambda}\right)_{TY,TZ}^{P}\mathscr{C}\left(X,PTY\right)}{rrrr}\ar[d]^{\mathscr{C}\left(Y,PTZ\right)\left(-\right)_{X,Y}^{\lambda}} &  &  &  & \mathscr{C}\left(PTY,PTZ\right)\mathscr{C}\left(X,PTY\right)\ar[d]^{\circ}\\
\mathscr{C}\left(Y,PTZ\right)\mathscr{C}\left(TX,PTY\right)\ar[d]^{\left(\left(-\right)_{Y,Z}^{\lambda}\right)_{TY,TZ}^{P}\mathscr{C}\left(TX,PTY\right)} &  & \dtwocell[0.7]{r}{\left(-\right)_{X,Y,Z}^{\Psi}} & \; & \mathscr{C}\left(X,PTZ\right)\ar[d]^{\left(-\right)_{X,Z}^{\lambda}}\\
\mathscr{C}\left(PTY,PTZ\right)\mathscr{C}\left(TX,PTY\right)\myard{\circ}{rrrr} &  &  &  & \mathscr{C}\left(TX,PTZ\right)
}
\]
satisfying (labeling only the cells due to space constraints)\footnote{On components these axioms reduce to the earlier no-iteration versions.
Moreover, it should be straightforward to recover the full diagrams
if one substitutes the three morphisms of pasting operators in full
detail, though some composites of pasting operators must be decomposed.} the triangle equation;
\begin{equation}
\tag{{P1}}\xymatrix@=1em{\bullet\ar[d]\ar@/^{1pc}/[rrrdrdrrrr]\ar@/_{1.2pc}/[ddddd] &  & \;\\
\bullet\ar[rrrdr]\ar[ddd]\ar[rrr] &  & \;\dtwocell[0.4]{u}{\left(-\right)_{Y,Z}^{\psi}} & \bullet\ar[rrrrrd]\\
\dtwocell[0.4]{dr}{\mathscr{C}\left(Y,PTZ\right)\cdot\left(-\right)_{Y,Y}^{\xi}} &  &  &  & \bullet\ar[rr]\ar[d] &  & \bullet\ar[rr]\dtwocell[0.5]{lld}{\left(-\right)_{Y,Y,Z}^{\Psi}} &  & \bullet\ar[rr] &  & \bullet & = & \textnormal{id}_{\left(-\right)_{Y,Z}^{\lambda}}\\
 & \; & \; &  & \bullet\ar[rr] & \; & \bullet\ar[urrrr]\\
\bullet\ar[rrrur]\ar[d]\\
\bullet\ar@/_{1.2pc}/[uuurrrrrrrrrr]
}
\label{P1}
\end{equation}
the pentagon equation;
\begin{equation}
\tag{{P2}}\xymatrix@=1em{ &  &  &  & \bullet\ar[rrrr]\dtwocell[0.5]{d}{\left(-\right)_{Y,Z,W}^{\Psi}\cdot\mathscr{C}\left(X,PTY\right)} &  &  &  & \bullet\ar[rrd]\\
\bullet\ar@{=}[d]\ar[rrrr]\ar[rrrru] &  &  &  & \bullet\ar[rrr] &  & \; & \bullet\ar[rrr] &  &  & \bullet\ar[dd]\\
\bullet\ar[rrrrrr]\ar[d] &  &  &  &  &  & \bullet\ar[d]\ar[rrr] &  &  & \bullet\ar[d]\ar[rd]\\
\bullet\ar[d] & \dtwocell[0.3]{l}{\mathscr{C}\left(Z,PTW\right)\cdot\left(-\right)_{X,Y,Z}^{\Psi}} &  &  &  &  & \bullet\ar[d]\ar[rrr] &  &  & \bullet\ar[d] & \bullet\ar[ld]\\
\bullet\ar[rrrrrr]\ar[d] &  &  &  &  &  & \bullet\ar[d] & \dtwocell[0.4]{l}{\left(-\right)_{X,Z,W}^{\Psi}} &  & \bullet\ar[d]\\
\bullet\ar[rrrrrr] &  &  &  &  &  & \bullet\ar[rrr] &  &  & \bullet
}
\label{P2}
\end{equation}
is equal to
\[
\xymatrix@=1em{\bullet\ar[d]\ar[rrr] &  &  & \bullet\ar[rrr] &  &  & \bullet\ar[d]\ar[rrr] &  &  & \bullet\ar[d]\\
\dtwocell[0.6]{rd}{\left(-\right)_{Y,Z,W}^{\Psi}\cdot\mathscr{C}\left(TX,PTY\right)}\bullet\ar[d]\ar[rrr] &  &  & \bullet\ar[rrr] &  &  & \bullet\ar[d] &  &  & \bullet\ar[dd]\\
\bullet\ar[d]\ar[rrr] & \; &  & \bullet\ar[rrr] &  &  & \bullet\ar[d] & \dtwocell[0.45]{ul}{\left(-\right)_{X,Y,W}^{\Psi}}\\
\bullet\ar[rrrrrr]\ar[rrrrrrd] &  &  &  &  &  & \bullet\ar[rrr] &  &  & \bullet\\
 &  &  &  &  &  & \bullet\ar[rrru]
}
\]
compatibility of $\left(-\right)^{\psi}$ with $\mathscr{C\to\mathscr{C}}_{P}$;
\begin{equation}
\tag{{PC1}}\xymatrix@=1em{\bullet\ar[dd]\ar[rrrr] &  & \dtwocell[0.5]{ddll}{\left(-\right)_{X,Y}^{\xi}} &  & \bullet\ar[d]\ar@/^{1.1pc}/[rrrdd]\\
 &  &  &  & \bullet\ar[rrd] & \dtwocell[0.1]{dl}{\left(-\right)_{X,Y}^{\psi}} &  &  & = & \textnormal{id}_{\left(-\right)_{X,TY}^{\eta}}\\
\bullet\ar[rr]\ar@/_{1pc}/[rrrrrrd] &  & \bullet\ar[rur]\ar[rr] &  & \bullet\ar[rr]\ar[rrd] &  & \bullet\ar[r] & \bullet\\
 &  &  &  &  &  & \bullet\ar[ur]
}
\label{PC1}
\end{equation}
compatibility of $\left(-\right)^{\Psi}$ with $\mathscr{C\to\mathscr{C}}_{P}$;
\begin{equation}
\tag{{PC2}}\xymatrix@=1em{\bullet\ar[rr]\ar[d] &  & \bullet\ar[rrrr]\ar[d] & \dtwocell[0.5]{dl}{\left(Y,PTZ\right)\left(-\right)_{X,Y}^{\xi}} &  &  & \bullet\myar{}{rrrr}\ar[d] &  &  &  & \bullet\ar[d]\\
\bullet\ar[r]\ar@/_{0pc}/[dd] & \bullet\ar[rd]_{}\ar@/_{0.4pc}/[rrrdrru]|-{} & \bullet\ar[rrrr]\dtwocell[0.6]{d}{\left(-\right)_{Y,Z}^{\xi}\left(TX,PTY\right)} &  & \; &  & \bullet\ar[d] & \; & \dtwocell[0.5]{ldl}{\left(-\right)_{X,Y,Z}^{\Psi}} &  & \bullet\ar[dd]\\
 &  & \bullet\ar[rrrr]\ar[d] &  & \; &  & \bullet\ar[d] & \;\\
\bullet\ar[rr]\ar@/_{1pc}/[rrrrrrrrrr] &  & \bullet\ar[rrrr] &  &  &  & \bullet\myard{}{rrrr} &  &  &  & \bullet
}
\label{PC2}
\end{equation}
is equal to
\[
\xymatrix@=1em{ & \bullet\ar[r] & \bullet\dtwocell[0.6]{d}{\left(-\right)_{Y,Z}^{\xi}\left(X,PTY\right)}\ar@/^{0pc}/[rrrrrud] & \; &  &  &  & \bullet\ar[r] & \bullet\ar[r] & \bullet\ar[dd]\\
\bullet\ar[ur]\ar[r]\ar@/_{0.5pc}/[drdrrrruurdl] & \bullet\ar[ur]\ar[rrrr] & \; & \; &  & \bullet\ar@/_{0pc}/[rru] &  & \bullet\ar[rur]\ar[d]\dtwocell[0.2]{drr}{\left(-\right)_{X,Z}^{\xi}} &  & \;\\
 &  &  &  &  & \bullet\ar[r]\ar[rru]\ar[rruu] & \bullet\ar[r] & \bullet\ar[rr] &  & \bullet
}
\]
\end{itemize}
\end{defn}

\begin{rem}
The reason for giving the coherence axioms here (as opposed to just
relying on the no-iteration form) is to remind the reader how the
conditions of pseudodistributive laws correspond to bicategorical
coherence axioms.
\end{rem}

\begin{rem}
Warpings in the context of wreaths have been considered by Chikhladze
\cite{Chk}. This should give a realistic approach to finding the
coherence axioms of pseudo-wreaths. Furthermore, there are structures
between distributive laws and wreaths which may be of interest. For
instance, the above formulation with the compatibility data\footnote{The loss of compatibility data may break the bijectivity of the change
of variables $\lambda\mapsto\alpha$.} and axioms omitted is more general than a distributive law, less
general than a wreath, and has a simpler presentation than both.
\end{rem}

\section{Equivalence of presentations of pseudodistributive laws\label{2dproof}}

As all five of our definitions of pseudodistributive laws are new,
we must justify them by showing they are equivalent to a pseudodistributive
law in the sense of Marmolejo \cite{marm1999}. This is the reason
for proving the following theorem, which makes use of the equivalence
between Marmolejo's definition of pseudodistributive law and compatible
extensions of a pseudomonad to the Kleisli bicategory shown in \cite{cheng2003}.
\begin{thm}
\label{pseudodistequiv} Given two pseudomonads $\left(T,u,m\right)$
and $\left(P,\eta,\mu\right)$ on a 2-category $\mathscr{C}$, the
following are in equivalence:
\begin{enumerate}
\item a pseudodistributive law $\lambda\colon TP\to PT$ in pseudomonoidal
form;
\item a pseudodistributive law $\lambda\colon TP\to PT$ in Kleisli-decagon
form;
\item a pseudodistributive law $\alpha\colon TPT\to PT$ in pseudoalgebra
form;
\item a pseudodistributive law $\left(-\right)^{\lambda}\colon\mathscr{C}\left(-,PT-\right)\to\mathscr{C}\left(T-,PT-\right)$
in no-iteration form;
\item a pseudodistributive law $\left(-\right)^{\lambda}\colon\mathscr{C}\left(-,PT-\right)\to\mathscr{C}\left(T-,PT-\right)$
in warping form;
\item an extension of $\left(T,u,m\right)$ to a pseudomonad on the Kleisli
bicategory of $\left(P,\eta,\mu\right)$ which is compatible with
the Kleisli pseudofunctor $\mathscr{C}\to\mathscr{C}_{P}$.
\end{enumerate}
\end{thm}

We will give a sketch proof of this theorem by constructing functors
as in the below diagram
\begin{equation}
\xymatrix@R=1em{\lambda\textnormal{-8}\ar[r]^{\left(i\right)} & \lambda\textnormal{-5}\ar[r]\ar[r]^{\left(ii\right)} & \text{\ensuremath{\lambda}}\textnormal{-dec}\ar[d]^{\left(iii\right)}\\
 &  & \alpha\textnormal{ ps-alg}\ar[r]^{\left(iv\right)}\ar[d]^{\left(v\right)} & \left(-\right)^{\lambda}\textnormal{ no-it}\myar{\left(vi\right)}{r} & \left(-\right)^{\lambda}\textnormal{ warp}\\
\widetilde{T}\textnormal{ monoidal}\ar[rr]_{\textnormal{equiv}}\ar[uu]^{\textnormal{equiv}} &  & \widetilde{T}\textnormal{ no-iteration}
}
\label{main}
\end{equation}
and explaining why these are all equivalences. The top row lists Marmolejo
and Wood's 8-axiom presentation, our 5-axiom reduction and the decagon
presentation. The middle row lists the pseudoalgebra, no-iteration
and warping presentations. The bottom row references compatible extensions
of $T$ to the Kleisli bicategory of $P$, presented in monoidal or
no-iteration form. We label by ``equiv'' those functors which are
well known to be equivalences by results of Marmolejo and Wood \cite{marm2008,NoIteration},
and $\left(i\right)$ to $\left(vi\right)$ the remaining functors
we must define and show are equivalences.
\begin{rem}
We will not burden this paper with the definitions of morphisms of
pseudodistributive laws, as these are simply modifications $\lambda\Rightarrow\lambda'$
or $\alpha\Rightarrow\alpha'$ such that the obvious pasting diagrams
agree.
\end{rem}

\begin{rem}
The pseudoalgebra, no-iteration and warping formulations are in a
sense all the same definition but with different notations, and so
all have a similar set of four axioms. It is only a non-trivial change
of variables which causes a non-trivial change of axioms.
\end{rem}

\begin{rem}
Given the equivalence of our five formulations of pseudodistributive
law, one can show any reasonable coherence axiom must hold. This is
since any such condition may be rewritten in terms of the standard
pseudomonoidal form, and then shown directly. For example, to verify
one has a pseudoalgebra structure on $\alpha$, meaning we have coherent
isomorphisms $\alpha\cdot mPT\cong\alpha\cdot T\alpha$, one may rewrite
$\alpha$ in terms of $\lambda$ and use the standard axioms of the
pseudomonoidal form.
\end{rem}

In proving coherence conditions concerning pseudomonads and their
data, it is worth remembering that MacLane and Par\'e's coherence
theorem for bicategories \cite{MacPare} applies. This is a direct
consequence of the warping formulation of pseudomonads \cite{mndwarp},
which shows that pseudomonads may be expressed in terms of the data
of the resulting Kleisli bicategory. In the case of interest, the
coherence theorem reduces to the following.
\begin{prop}
Suppose that $T$ is a pseudomonad on any 2-category $\mathscr{C}$.
Then any two pasting diagrams (which are formal in the sense of the
following Remark) constructed from the pseudomonad data with the same
outside must be equal.
\end{prop}

\begin{rem}
Just as in MacLane's theorem, we will need an analogue of a formal
path:
\begin{itemize}
\item (pseudomonoidal form) here formal means it is constructed entirely
from pastings of naturality squares of $u$ and $m$ and the three
given modifications. Any unexpected equalities involving $u$, $m$
and the remaining data (such as the $mT=Tm$ in Isbell's counterexample)
cannot be used in the pasting diagram, even if they are true;
\item (no-iteration form) here formal means it is constructed entirely from
pastings of the given unit and multiplication cells. Any rewriting
of $f$ as $f'$ is not allowed, even if $f=f'$;
\end{itemize}
\end{rem}

We also note that if this proposition was not true, then the definition
of pseudomonad would need extra coherence axioms.

\subsection{A restricted equivalence of the pseudomonoidal and decagon forms}

The following lemma sets the stage for proving the equivalence $\left(ii\right)$
of the monoidal and decagon definitions by directly showing a restricted
version of this equivalence (meaning a number of axioms are left out
of both formulations).
\begin{rem}
For checking equalities of pasting diagrams as below, it is helpful
to note that given two diagrams constructed entirely from the same
data (for example if we are given two pasting diagrams both constructed
from two $\omega_{2}$ cells and one $\Omega$ cell), then showing
they are equal is typically a matter of applying naturality (middle
four interchange). 
\end{rem}

\begin{lem}
\label{ombij} For a given $\lambda$, the data $\left(\omega_{1},\omega_{2},\omega_{3},\omega_{4}\right)$
with axioms \eqref{W1} and \eqref{W2} is in bijection with the data
$\left(\omega_{1},\omega_{2},\Omega\right)$ with axiom \eqref{D1}. 
\end{lem}

\begin{proof}
From the modifications comprising the pentagons $\omega_{3}$ and
$\omega_{4}$, the decagon $\Omega$ is constructed as the pasting
diagram

\[
\xymatrix{TPTPT\ar[r]^{TP\lambda T}\ar[d]_{\lambda TPT} & TP^{2}T^{2}\ar[r]^{TP^{2}m}\ar[d]_{\lambda PT^{2}} & TP^{2}T\ar[r]^{T\mu T}\ar[d]_{\lambda PT} & TPT\ar[dd]^{\lambda T}\\
PT^{2}PT\ar[dd]_{PmPT}\ar[r]_{PT\lambda T} & PTPT^{2}\ar[d]_{P\lambda T^{2}}\ar[r]^{PTPm} & PTPT\ar[d]_{P\lambda T}\dtwocell[0.45]{r}{\omega_{4}T} & \;\\
\dtwocell[0.3]{r}{P\omega_{3}T} & P^{2}T^{3}\ar[d]_{P^{2}mT}\ar[r]^{P^{2}Tm} & P^{2}T^{2}\ar[d]_{P^{2}m}\ar[r]_{\mu T^{2}} & PT^{2}\ar[d]^{Pm}\\
PTPT\ar[r]_{P\lambda T} & P^{2}T^{2}\ar[r]_{P^{2}m} & P^{2}T\ar[r]_{\mu T} & PT
}
\]
Conversely, given the decagon $\Omega$ one recovers the pentagon
$\omega_{4}$ as
\[
\resizebox{1\textwidth}{!}{\xymatrix{ &  & \; &  &  & TP\ar@/^{0pc}/[d]^{TPu}\ar@/^{0.5pc}/[dr]^{\lambda}\\
TP^{2}\ar[r]^{TP^{2}u}\ar[d]_{\lambda P}\ar@/^{2.7pc}/[rrrrru]^{T\mu} & TP^{2}T\ar[r]^{TPuPT}\ar@/^{1.8pc}/[rr]^{TP^{2}uT}\ar[d]_{\lambda PT}\ar@/^{3.2pc}/[rrr]^{\textnormal{id}} & TPTPT\ar[r]^{TP\lambda T}\ar[d]_{\lambda TPT}\dtwocell[0.5]{rrrdd}{\Omega}\dtwocell[0.35]{u}{TP\omega_{1}T} & TP^{2}T^{2}\ar[r]^{TP^{2}m} & TP^{2}T\ar[r]^{T\mu T} & TPT\ar[d]^{\lambda T} & PT\ar@/^{0pc}/[dl]^{PTu}\ar@/^{1.5pc}/[ldd]^{\textnormal{id}}\\
PTP\ar[r]^{PTPu}\ar@/_{0pc}/[rd]_{P\lambda} & PTPT\ar[r]^{PTuPT}\ar@/_{0.5pc}/[rd]_{\textnormal{id}} & PT^{2}PT\ar[d]_{PmPT} &  &  & PT^{2}\ar[d]^{Pm}\\
 & P^{2}T\ar@/_{1.2pc}/[rr]_{P^{2}Tu}\ar@/_{2.5pc}/[rrr]_{\textnormal{id}} & PTPT\ar[r]_{P\lambda T} & P^{2}T^{2}\ar[r]_{P^{2}m} & P^{2}T\ar[r]_{\mu T} & PT
}
}
\]
and the pentagon $\omega_{3}$ as
\[
\resizebox{1\textwidth}{!}{\xymatrix{ & TPT\ar[r]^{TPTu}\ar@/^{1.5pc}/[rr]^{\textnormal{id}} & TPT^{2}\ar[rd]^{T\eta PT^{2}}\ar[r]^{TPm} & TPT\ar[rd]^{T\eta PT}\ar@/^{1pc}/[rrd]^{\textnormal{id}}\\
T^{2}P\ar[r]^{T^{2}Pu}\ar[d]_{mP}\ar[ur]^{T\lambda} & T^{2}PT\ar[r]^{T\eta TPT}\ar@<+0ex>@/_{0pc}/[rd]_{\eta T^{2}PT}\ar[d]_{mPT}\ar[ur]^{T\lambda T} & TPTPT\ar[r]^{TP\lambda T}\ar[d]^{\lambda TPT}\dtwocell[0.5]{rrrdd}{\Omega} & TP^{2}T^{2}\ar[r]^{TP^{2}m} & TP^{2}T\ar[r]^{T\mu T} & TPT\ar[d]^{\lambda T}\\
TP\ar[r]_{TPu}\ar[d]_{\lambda}\dtwocell[0.72]{urr}{\omega_{2}TPT} & TPT\ar[dr]_{\eta TPT}\ar[d]_{\lambda T} & PT^{2}PT\ar[d]^{PmPT} &  &  & PT^{2}\ar[d]^{Pm}\\
PT\ar[r]_{PTu}\ar@/_{0.5pc}/[drr]_{\textnormal{id}} & PT^{2}\ar@/_{1pc}/[rr]_{\eta PT^{2}}\ar[dr]_{Pm} & PTPT\ar[r]^{P\lambda T} & P^{2}T^{2}\ar[r]^{P^{2}m} & P^{2}T\ar[r]^{\mu T} & PT\\
 &  & PT\ar[urr]_{\eta PT}\ar@/_{1pc}/[rrru]_{\textnormal{id}}
}
}
\]
Now given $\left(\omega_{1},\omega_{2},\omega_{3},\omega_{4}\right)$
with axioms \eqref{W1} and \eqref{W2}, we may start with the left
hand side of \eqref{D1} and substitute the formula for the decagon
$\Omega$ in terms of $\omega_{3}$ and $\omega_{4}$. The $\omega_{2}$
cell is moved by naturality to the $\omega_{4}$ cell and both are
eliminated by \eqref{W2}, and similarly the $\omega_{1}$ cell and
$\omega_{3}$ cell are eliminated by \eqref{W1}, leaving the identity
or right hand side of \eqref{D1}. Replacing $\Omega$ by its formula
in terms of $\omega_{3}$ and $\omega_{4}$ in the above two diagrams,
we recover $\omega_{3}$ and $\omega_{4}$ by the same process of
moving cells by naturality and eliminating.

Conversely, given $\left(\omega_{1},\omega_{2},\Omega\right)$ with
axiom \eqref{D1}. We may start with the left hand side of \eqref{W1}
and substitute the formula for $\omega_{3}$ in terms of $\Omega$.
This gives a decagon $\Omega$ with a $\omega_{2}$ cell on the left,
and a $\omega_{1}$ cell on the top. This should already look similar
to axiom \eqref{D1}, which can be applied with a minor rearrangement.
The proof of \eqref{W2} is similar. Writing the formula for $\Omega$
in terms of $\omega_{3}$ and $\omega_{4}$ and then substituting
the above formulas for $\omega_{3}$ and $\omega_{4}$ in terms of
$\Omega$, one may move the $\omega_{1}$ cell by naturality to apply
\eqref{D1} leaving only $\Omega$ after some simplification.
\end{proof}

\subsection{Decagons to pseudoalgebras to no-iteration forms}

We now define $\left(iii\right)$ and $\left(v\right)$, and show
these are equivalences. The construction $\left(iii\right)$ is the
following simple rewriting.
\begin{prop}
A pseudodistributive law $\left(\lambda,\omega_{1},\omega_{2},\Omega\right)$
in decagon form with axiom \eqref{W10} gives rise to a pseudodistributive
law $\left(\alpha,\psi,\xi,\Psi\right)$ in pseudoalgebra form. This
also holds with the compatibility axioms \eqref{W10}, \eqref{DC2}
and \eqref{MC1}, \eqref{MC2} removed respectively.
\end{prop}

\begin{proof}
Given a $\lambda\colon TP\to PT$ we define $\alpha\colon TPT\to PT$
as the composite
\[
\xymatrix{TPT\ar[r]^{\lambda T} & PT^{2}\ar[r]^{Pm} & PT}
\]
Fortunately, the decagon $\Omega$ and axiom \eqref{D2a} is constructed
from such composites. Thus we can simply take $\Psi=\Omega$, and
\eqref{M2} is just \eqref{D2a} with these composites relabeled as
$\alpha$. The modifications $\psi$ and $\xi$ are given by
\[
\xymatrix@=1em{TPT\myar{\lambda T}{r}\ltwocell[0.35]{dr}{\omega_{1}T} & PT^{2}\ar[r]^{Pm} & PT &  &  & T^{2}\myar{T\eta T}{rr}\ar[dd]_{m}\ar@/_{0.5pc}/[rrd]_{\eta T^{2}} & \dltwocell[0.3]{d}{\omega_{2}T} & TPT\ar[d]^{\lambda T}\\
 & \; &  &  &  &  & \; & PT^{2}\ar[d]^{Pm}\\
PT\ar[uu]^{uPT}\ar[ruu]_{PuT}\ar@/_{0.75pc}/[rruu]_{\textnormal{id}} &  &  &  &  & T\myard{\eta T}{rr} &  & PT
}
\]
respectively. The reader will notice these both appearing in axiom
\eqref{D1}, thus by relabeling we recover axiom \eqref{M1}.

For compatibility, we note the left of \eqref{MC1} expressed in terms
of the data $\left(\lambda,\omega_{1},\omega_{2},\Omega\right)$ is
\[
\resizebox{0.4\textwidth}{!}{\xymatrix@=1.5em{ &  & T\ar[dd]_{uT}\ar[rr]^{\eta T}\ar@/_{2pc}/[dddd]_{\textnormal{id}} &  & PT\ar[dd]_{uPT}\ar@/^{6pc}/[dddd]^{\textnormal{id}}\ar@/^{3pc}/[ddd]^{PuT}\\
 &  &  &  & \; & \dltwocell[0.6]{dl}{\omega_{1}T}\\
 &  & T^{2}\myar{T\eta T}{rr}\ar[dd]_{m}\ar@/_{0.5pc}/[rrd]_{\eta T^{2}} & \dltwocell[0.3]{d}{\omega_{2}T} & TPT\ar[d]_{\lambda T}\\
 &  &  & \; & PT^{2}\ar[d]_{Pm}\\
 &  & T\myard{\eta T}{rr} &  & PT
}
}
\]
and it is a clear consequence of \eqref{W10} that this simplifies
to $\textnormal{id}_{\eta T}$. Similarly, by substitution the left
of \eqref{MC2} becomes
\[
\resizebox{0.95\textwidth}{!}{\xymatrix@=1.5em{ & T^{3}\myar{T\eta T^{2}}{rr}\ar[dd]_{mT}\ar@/_{0.5pc}/[rrd]_{\eta T^{3}} & \dltwocell[0.3]{d}{\omega_{2}T^{2}} & TPT^{2}\ar[rr]^{TPT\eta T}\ar[d]^{\lambda T^{2}} &  & TPTPT\ar[rr]^{TP\lambda T}\ar[d]^{\lambda TPT} &  & \dtwocell[0.5]{rrdd}{\Omega}TP^{2}T^{2}\ar[rr]^{TP^{2}m} &  & TP^{2}T\ar[rr]^{T\mu T} &  & TPT\ar[d]^{\lambda T}\\
 &  & \; & PT^{3}\ar[rr]^{PT^{2}\eta T}\ar[d]^{PmT} &  & PT^{2}PT\ar[d]^{PmPT} &  &  &  &  &  & PT^{2}\ar[d]^{Pm}\\
 & T^{2}\myard{\eta T^{2}}{rr}\ar@/_{0.3pc}/[rrd]_{m} &  & PT^{2}\ar[rr]_{PT\eta T}\ar@/_{1.5pc}/[rrrr]|-{P\eta T^{2}}\ar@/_{0.3pc}/[rrd]_{Pm} &  & PTPT\ar[rr]_{P\lambda T}\dtwocell[0.4]{d}{P\omega_{2}T} &  & P^{2}T^{2}\ar[rr]^{P^{2}m} &  & P^{2}T\ar[rr]^{\mu T} &  & PT\\
 &  &  & T\ar[rr]_{\eta T} &  & PT\ar[urrrr]_{P\eta T}\ar@/_{0.5pc}/[urrrrrr]_{\textnormal{id}}
}
}
\]
which by \eqref{DC2} reduces to
\[
\xymatrix@=1em{T^{3}\ar[rr]^{T\eta T^{2}}\ar@/_{3.2pc}/[rrrrrrrrdrr]_{\eta T^{3}}\ar[dd]_{mT} & \dltwocell[0.3]{rd}{\omega_{2}T^{2}} & TPT^{2}\ar[rr]^{TPT\eta T}\ar@/_{1.5pc}/[rrrr]|-{TP\eta T^{2}}\ar@/_{1.7pc}/[rrrrrdr]_{\textnormal{id}} &  & TPTPT\ar[rr]^{TP\lambda T}\dtwocell[0.4]{d}{TP\omega_{2}T} &  & TP^{2}T^{2}\ar[rr]^{TP^{2}m}\ar[rrd]_{T\mu T^{2}} &  & TP^{2}T\ar[rrrr]^{T\mu T} &  &  &  & TPT\ar[d]^{\lambda T}\\
 &  & \; &  & \; &  &  &  & TPT^{2}\ar[rrrru]^{TPm}\ar[rr]_{\lambda T^{2}} &  & PT^{3}\ar[rr]_{PTm}\ar[d]_{PmT} &  & PT^{2}\ar[d]^{Pm}\\
T^{2}\ar@/_{1.3pc}/[rrrrrrrrrr]_{\eta T^{2}}\ar@/_{1pc}/[rrrrd]_{m} &  &  &  &  &  &  &  &  &  & PT^{2}\ar[rr]^{Pm} &  & PT\\
 &  &  &  & T\ar@/_{1.4pc}/[rrrrrrrru]_{\eta T}
}
\]
and naturality of $m$ then gives 
\[
\xymatrix@=1em{T^{3}\ar[rr]^{T\eta T^{2}}\ar[ddd]_{mT}\ar[rrdd]_{Tm} &  & TPT^{2}\ar[rr]^{TPT\eta T}\ar@/_{1.5pc}/[rrrr]|-{TP\eta T^{2}}\ar[rrdd]_{TPm} &  & TPTPT\ar[rr]^{TP\lambda T}\dtwocell[0.4]{d}{TP\omega_{2}T} &  & TP^{2}T^{2}\ar[rr]^{TP^{2}m} &  & TP^{2}T\ar[rrrr]^{T\mu T} &  &  &  & TPT\ar[dd]^{\lambda T}\\
 &  &  &  & \;\\
 &  & T^{2}\ar[rr]^{T\eta T}\ar@/_{0pc}/[d]_{m}\ar@/_{2pc}/[rrrrrrrrrr]_{\eta T^{2}} &  & TPT\ar@/_{1pc}/[urrrurrrrr]_{\textnormal{id}}\ar[uurrrr]_{TP\eta T}\dtwocell[0.4]{d}{\omega_{2}T} &  &  &  &  &  &  &  & PT^{2}\ar[d]^{Pm}\\
T^{2}\ar@/_{0pc}/[rr]_{m} &  & T\ar@/_{1.2pc}/[rrrrrrrrrr]_{\eta T} &  &  &  &  &  &  &  &  &  & PT
}
\]
which is the right hand side of \eqref{MC2}.
\end{proof}
We now see how a pseudodistributive law $\left(\alpha,\psi,\xi,\Psi\right)$
in pseudoalgebra form gives rise to a pseudomonad on the Kleisli bicategory,
defining $\left(v\right)$.
\begin{prop}
A pseudodistributive law $\left(\alpha,\psi,\xi,\Psi\right)$ in pseudoalgebra
form gives rise to a compatible pseudomonad $\widetilde{T}$ extending
$T$ to the Kleisli bicategory of $P$. With the compatibility axioms
\eqref{MC1}, \eqref{MC2} removed we have a possibly non-compatible
extension.
\end{prop}

\begin{proof}
Suppose we are given a pseudodistributive law $\alpha\colon TPT\to PT$
in pseudoalgebra form. We will define a pseudomonad $\widetilde{T}$
in pseudoextensive form (as in Definition \ref{pseudomonadextensive})
on the Kleisli bicategory of $\left(P,\eta,\mu\right)$. We define
$\widetilde{T}$ to have the same action on objects as $T$. For each
$X\in\mathbf{Kl}\left(P\right)$, we take our unit $\widetilde{u}_{X}\colon X\rightsquigarrow TX$
to be the composite
\[
\xymatrix{X\ar[r]^{u_{X}} & TX\ar[r]^{\eta_{TX}} & PTX}
\]
Each functor $\mathbf{Kl}\left(P\right)\left(X,TY\right)\to\mathbf{Kl}\left(P\right)\left(TX,TY\right)$,
or more simply $\mathscr{C}\left(X,PTY\right)\to\mathscr{C}\left(TX,PTY\right)$\footnote{This operation may also be denoted by the pseudo-pasting operator
$\left(-\right)^{\lambda}$.}, is defined by sending an $f\colon X\to PTY$ to $\alpha Y\cdot Tf\colon TX\to PTY$.
For each $f\colon X\to PTY$ we take the 2-cell $\phi_{f}\colon f\Rightarrow f^{\widetilde{T}}\cdot\widetilde{u}_{X}$
as the pasting
\[
\xymatrix@=1em{X\ar[rd]_{uX}\ar[r]^{f} & PTY\ar[dr]^{uPTY}\ar@/^{0pc}/[rrr]^{\textnormal{id}} &  & \; & PTY\ar[rd]^{\eta PTY}\ar@/^{0pc}/[rr]^{\textnormal{id}} &  & PTY\\
 & TX\myar{Tf}{r}\ar[dr]_{\eta TX}\dtwocell[0.6]{urr}{\psi Y} & TPTY\ar@/_{0.5pc}/[rru]^{\alpha Y}\ar[rd]^{\eta TPTY} &  &  & P^{2}TY\ar[ru]_{\mu TY}\\
 &  & PTX\ar[r]_{PTf} & PTPTY\ar@/_{0.5pc}/[urr]_{P\alpha Y}
}
\]
for each $X$ we take the 2-cell $\theta_{X}\colon\left(\widetilde{u}_{X}\right)^{\widetilde{T}}\Rightarrow\textnormal{id}_{\widetilde{T}X}$
to be
\[
\xymatrix@=1em{TX\ar[rr]^{TuX}\ar@/_{0.2pc}/[rdr]_{\textnormal{id}} &  & T^{2}X\myar{T\eta TX}{rr}\ar[d]_{mX}\dltwocell[0.4]{drr}{\xi X} &  & TPTX\ar[d]^{\alpha X}\\
 &  & TX\myard{\eta TX}{rr} &  & PTX
}
\]
and for all $f\colon X\to PTY$ and $g\colon Y\to PTZ$ we take $\delta_{g,f}\colon\left(g^{\widetilde{T}}\cdot f\right)^{\widetilde{T}}\Rightarrow g^{\widetilde{T}}\cdot f^{\widetilde{T}}$
as
\[
\xymatrix{TX\myar{Tf}{r} & TPTY\myar{TPTg}{r}\ar[d]_{\alpha Y} & TPTPTZ\ar[d]_{\alpha PTZ}\myar{TP\alpha Z}{r} & TP^{2}TZ\ar[r]^{T\mu TZ}\dtwocell[0.5]{d}{\Psi Z} & TPTZ\ar[d]^{\alpha Z}\\
 & PTY\myard{PTg}{r} & PTPTZ\ar[r]_{P\alpha Z} & P^{2}TZ\ar[r]_{\mu TZ} & PTZ
}
\]
Note that technically there is no real choice of the pseudomonad data
of $\widetilde{T}$ here, as it is forced by the compatibility conditions
required for an extension to the Kleisli bicategory. Naturality is
clear in the above definitions. Moreover, the two axioms \eqref{M1}
and \eqref{M2} ensure the two coherence conditions of a pseudomonad
in no-iteration form are satisfied.

To ensure compatibility with the Kleisli pseudofunctor $\mathscr{C}\to\mathscr{C}_{P}$
we must provide an isomorphism in the square where both $T$ and $\widetilde{T}$
are defined extensively
\[
\xymatrix@=1em{\mathscr{C}\ar[rr]^{T}\ar[d] &  & \mathscr{C}\ar[d]\\
\mathscr{C}_{P}\ar[rr]_{\widetilde{T}} &  & \mathscr{C}_{P}
}
\]
This is the identity on objects, and for a map $k\colon X\to TY$
in $\mathscr{C}$ it is the isomorphism $\left(\eta TY\cdot k\right)^{\widetilde{T}}\cong\eta TY\cdot k^{T}$
denoted $\xi^{k}$ and constructed as
\[
\xymatrix@=1em{TX\ar[rr]^{Tk}\ar@/_{0.2pc}/[rdr]_{k^{T}} &  & T^{2}Y\myar{T\eta TY}{rr}\ar[d]_{mY}\dltwocell[0.4]{drr}{\xi Y} &  & TPTY\ar[d]^{\alpha Y}\\
 &  & TY\myard{\eta TY}{rr} &  & PTY
}
\]
That $\xi^{k}$ respects the unit $\phi$ is \eqref{MC1}, that $\xi^{k}$
respects the unit $\theta$ is that $\theta_{X}$ must be $\xi^{uX}$,
and finally that $\xi^{k}$ respects the associator $\delta$ is \eqref{MC2}.
\end{proof}

\subsection{Redundancy of pseudomonoidal coherence axioms\label{showred}}

From axioms \eqref{W1} through to \eqref{W5} follow the axioms \eqref{D1}
and \eqref{D2a}\footnote{Note the fact that axiom \eqref{D2a} involves only the decagons (which
are constructed from only the pentagons), and so it is routine to
check that \eqref{D2a} follows from the axioms \eqref{W3}, \eqref{W4},
\eqref{W5} which concern the coherence conditions of only pentagons.
Lemma \ref{ombij} explains \eqref{D1}.}, from which follow three redundant axioms (D3), (D4) and (D5). This
restricted version of the decagon formulation (without the condition
\eqref{DC2} for compatibility) is enough to construct the extension
and show redundancy of the three axioms \eqref{W8}, \eqref{W9} and
\eqref{W10}.

The three redundant axioms (D3), (D4) and (D5) come from considering
the redundant axioms of a pseudomonad in no-iteration form as in Remark
\ref{redundantremark} (being careful to understand these axioms in
the context of the extended pseudomonad $\widetilde{T}$ not $T$),
where the pseudomonad $\widetilde{T}$ is constructed from the decagon
formulation of a pseudodistributive law\footnote{This construction was detailed using the pseudoalgebra formulation,
but may be easily converted to the decagon formulation. Note also
that the construction of $\widetilde{T}$ by itself does not require
any compatibility conditions.}. The middle axiom of Remark \ref{redundantremark} (in the base case
with $f$ an identity) expressed in decagon form is then that
\begin{equation}
\tag{{D4}}\resizebox{0.95\textwidth}{!}{\xymatrix@=1.5em{TPT\ar[d]_{\lambda T}\ar[rr]^{TPTu} &  & TPT^{2}\ar[rr]^{TPT\eta T}\ar[d]^{\lambda T^{2}} &  & TPTPT\ar[rr]^{TP\lambda T}\ar[d]^{\lambda TPT} &  & \dtwocell[0.5]{rrdd}{\Omega}TP^{2}T^{2}\ar[rr]^{TP^{2}m} &  & TP^{2}T\ar[rr]^{T\mu T} &  & TPT\ar[d]^{\lambda T}\\
PT^{2}\ar[rr]^{PT^{2}u}\ar[d]_{Pm} &  & PT^{3}\ar[rr]^{PT^{2}\eta T}\ar[dd]^{PmT} &  & PT^{2}PT\ar[d]^{PmPT} &  &  &  &  &  & PT^{2}\ar[d]^{Pm}\\
PT\ar[rdr]^{PTu}\ar@/_{5pc}/[rrrrrrrrrr]|-{\textnormal{id}} &  &  &  & PTPT\ar[rr]_{P\lambda T}\dtwocell[0.6]{d}{P\omega_{2}T} &  & P^{2}T^{2}\ar[rr]^{P^{2}m}\ar[rrd]_{\mu T^{2}} &  & P^{2}T\ar[rr]^{\mu T} &  & PT\\
 &  & PT^{2}\ar[rur]_{PT\eta T}\ar@/_{1pc}/[rrrur]_{P\eta T^{2}}\ar@/_{1.2pc}/[rrrrrr]|-{\textnormal{id}} &  & \; &  &  &  & PT^{2}\ar[rru]^{Pm}
}
}\label{D4}
\end{equation}
must be equal to
\[
\xymatrix@=1em{ &  &  & TPTPT\ar@/^{0.3pc}/[rd]^{TP\lambda T}\dtwocell[0.6]{d}{TP\omega_{2}T} &  &  & TP^{2}T\ar[rrd]^{T\mu T}\\
TPT\ar[rr]^{TPTu}\ar@/_{2pc}/[rrrrrrrr]|-{\textnormal{id}} &  & TPT^{2}\ar[ur]^{TPT\eta T}\ar@/_{0pc}/[rr]_{TP\eta T^{2}}\ar@/_{1.3pc}/[rrrr]|-{\textnormal{id}} & \; & TP^{2}T^{2}\ar[rr]^{T\mu T^{2}}\ar@/^{0.3pc}/[rru]^{TP^{2}m} &  & TPT^{2}\ar[rr]^{TPm} &  & TPT\ar[r]^{\lambda T} & PT^{2}\ar[r]^{Pm} & PT
}
\]
Following a similar calculation, the leftmost redundant axiom of Remark
\ref{redundantremark} works out to be \eqref{W10} which we also
call (D5), and the rightmost axiom becomes (D3) which we have not
stated since it is not strictly necessary for the proof. In particular,
this gives the following.
\begin{prop}
\label{d5red} In the decagon presentation the axiom \eqref{W10}
is redundant\footnote{The proof of this proposition requires first knowing how to construct
the extension, justifying the ordering.}.
\end{prop}

The axiom \eqref{W10} must also be redundant in the pseudomonoidal
presentation, since the decagon formulation may be constructed from
it. The proof of the redundancy of \eqref{W8} and \eqref{W9} is
slightly more work, so we give it below.
\begin{prop}
In the pseudomonoidal presentation the axioms \eqref{W8}, \eqref{W9}
and \eqref{W10} are redundant and $\left(i\right)$ defines an isomorphism.
\end{prop}

\begin{proof}
Instead of giving the argument for both remaining axioms, we will
show the redundancy of $\eqref{W8}$ in some detail. The argument
for $\eqref{W9}$ is then similar.

Given the left hand side of \eqref{W8}, we start by substituting
the formula for the pentagon $\omega_{3}$ in terms of the decagon
$\Omega$ as in Lemma \ref{ombij} giving
\[
\resizebox{1\textwidth}{!}{\xymatrix{ &  &  & TPT\ar[r]^{TPTu}\ar@/^{1.5pc}/[rr]^{\textnormal{id}} & TPT^{2}\ar[rd]^{T\eta PT^{2}}\ar[r]^{TPm} & TPT\ar[rd]^{T\eta PT}\ar@/^{1pc}/[rrd]^{\textnormal{id}}\\
T^{2}\ar[rr]^{T^{2}\eta}\ar[rd]_{m} &  & T^{2}P\ar[r]^{T^{2}Pu}\ar[d]_{mP}\ar[ur]^{T\lambda} & T^{2}PT\ar[r]^{T\eta TPT}\ar@<+0ex>@/_{0pc}/[rd]_{\eta T^{2}PT}\ar[d]_{mPT}\ar[ur]^{T\lambda T} & TPTPT\ar[r]^{TP\lambda T}\ar[d]^{\lambda TPT}\dtwocell[0.5]{rrrdd}{\Omega} & TP^{2}T^{2}\ar[r]^{TP^{2}m} & TP^{2}T\ar[r]^{T\mu T} & TPT\ar[d]^{\lambda T}\\
 & T\ar[r]_{T\eta}\ar@/_{0.5pc}/[rd]_{\eta T} & TP\ar[r]_{TPu}\ar[d]_{\lambda}\dtwocell[0.72]{urr}{\omega_{2}TPT} & TPT\ar[dr]_{\eta TPT}\ar[d]_{\lambda T} & PT^{2}PT\ar[d]^{PmPT} &  &  & PT^{2}\ar[d]^{Pm}\\
 & \dtwocell[0.5]{ur}{\omega_{2}} & PT\ar[r]_{PTu}\ar@/_{0.5pc}/[drr]_{\textnormal{id}} & PT^{2}\ar@/_{1pc}/[rr]_{\eta PT^{2}}\ar[dr]_{Pm} & PTPT\ar[r]^{P\lambda T} & P^{2}T^{2}\ar[r]^{P^{2}m} & P^{2}T\ar[r]^{\mu T} & PT\\
 &  &  &  & PT\ar[urr]_{\eta PT}\ar@/_{1pc}/[rrru]_{\textnormal{id}}
}
}
\]
We move the bottom $\omega_{2}$ along naturality squares of $u$
and $\eta$, so that we are more ready to apply axiom \eqref{D4},
giving

\[
\resizebox{1\textwidth}{!}{\xymatrix{ &  &  & TPT\ar[r]^{TPTu}\ar@/^{1.5pc}/[rr]^{\textnormal{id}} & TPT^{2}\ar[rd]^{T\eta PT^{2}}\ar[r]^{TPm} & TPT\ar[rd]^{T\eta PT}\ar@/^{1pc}/[rrd]^{\textnormal{id}}\\
T^{2}\ar[rr]^{T^{2}\eta}\ar[d]_{m} &  & T^{2}P\ar[r]^{T^{2}Pu}\ar[d]_{mP}\ar[ur]^{T\lambda} & T^{2}PT\ar[r]^{T\eta TPT}\ar@<+0ex>@/_{0pc}/[rd]_{\eta T^{2}PT}\ar[d]_{mPT}\ar[ur]^{T\lambda T} & TPTPT\ar[r]^{TP\lambda T}\ar[d]^{\lambda TPT}\dtwocell[0.35]{rrrdd}{\Omega} & TP^{2}T^{2}\ar[r]^{TP^{2}m} & TP^{2}T\ar[r]^{T\mu T} & TPT\ar[d]^{\lambda T}\\
T\ar[rr]_{T\eta}\ar[drr]_{Tu}\ar[rddr]_{\eta T} &  & TP\ar[r]_{TPu}\dtwocell[0.72]{urr}{\omega_{2}TPT} & TPT\ar[dr]_{\eta TPT} & PT^{2}PT\ar[d]^{PmPT} &  & P^{2}T\ar[dr]^{\mu T} & PT^{2}\ar[d]^{Pm}\\
 &  & T^{2}\ar[ur]_{T\eta T}\ar[rd]_{\eta T^{2}} &  & PTPT\ar[r]^{P\lambda T}\dtwocell[0.3]{d}{P\omega_{2}T} & P^{2}T^{2}\ar[ru]^{P^{2}m}\ar[r]_{\mu T^{2}} & PT^{2}\ar[r]_{Pm} & PT\\
 &  & PT\ar@/_{0pc}/[r]^{PTu}\ar@/_{2pc}/[rrrrru]_{\textnormal{id}} & PT^{2}\ar[ur]^{PT\eta T}\ar@/_{0.3pc}/[rur]_{P\eta T^{2}}\ar@/_{1pc}/[rurr]_{\textnormal{id}} & \;\\
\\
}
}
\]
and we also restrict the $\omega_{2}TPT$ to a $\omega_{2}T$ to look
more like the right hand side of \eqref{W8}, giving

\[
\resizebox{1\textwidth}{!}{\xymatrix{ &  &  & TPT\ar[dr]^{TPTu}\ar@/^{0.5pc}/[rrd]^{\textnormal{id}}\\
 & T^{2}P\ar[rur]^{T\lambda}\ar[rrud]^{T^{2}Pu}\ar[rrd]_{T\eta TP} &  & T^{2}PT\ar[dr]^{T\eta TPT}\ar[udr]^{T\lambda T} & TPT^{2}\ar[rd]^{T\eta PT^{2}}\ar[r]^{TPm} & TPT\ar[rd]^{T\eta PT}\ar@/^{1pc}/[rrd]^{\textnormal{id}}\\
T^{2}\ar[ru]^{T^{2}\eta}\ar[d]_{m}\ar@/_{0.7pc}/[rrdd]_{\eta T^{2}}\ar[rrd]_{T\eta T} &  &  & TPTP\ar[r]_{TPTPu} & TPTPT\ar[r]^{TP\lambda T}\ar[d]^{\lambda TPT}\dtwocell[0.35]{rrrdd}{\Omega} & TP^{2}T^{2}\ar[r]^{TP^{2}m} & TP^{2}T\ar[r]^{T\mu T} & TPT\ar[d]^{\lambda T}\\
T\ar[rddr]_{\eta T} & \dtwocell[0.25]{dr}{\omega_{2}T} & TPT\ar[d]_{\lambda T}\ar[ur]_{TPT\eta}\ar[r]_{TPTu} & TPT^{2}\ar[ru]_{TPT\eta T}\ar[d]_{\lambda T^{2}} & PT^{2}PT\ar[d]^{PmPT} &  & P^{2}T\ar[dr]^{\mu T} & PT^{2}\ar[d]^{Pm}\\
 & \; & PT^{2}\ar[d]_{Pm}\ar[r]^{PT^{2}u} & PT^{3}\ar[d]_{PmT}\ar[ru]^{PT^{2}\eta T} & PTPT\ar[r]^{P\lambda T}\dtwocell[0.3]{d}{P\omega_{2}T} & P^{2}T^{2}\ar[ru]^{P^{2}m}\ar[r]_{\mu T^{2}} & PT^{2}\ar[r]_{Pm} & PT\\
 &  & PT\ar@/_{0pc}/[r]^{PTu}\ar@/_{2pc}/[rrrrru]_{\textnormal{id}} & PT^{2}\ar[ur]^{PT\eta T}\ar@/_{0.3pc}/[rur]_{P\eta T^{2}}\ar@/_{1pc}/[rurr]_{\textnormal{id}} & \;
}
}
\]
Now applying axiom \eqref{D4} the above reduces to
\[
\resizebox{1\textwidth}{!}{\xymatrix{ & TPT\ar@/^{0.7pc}/[rrr]^{TPTu}\ar@/^{2.3pc}/[rrrrrr]^{\textnormal{id}} &  & T^{2}PT\ar[rd]_{T\eta TPT}\ar[r]_{T\lambda T} & TPT^{2}\ar@/^{1pc}/[rdd]^{T\eta PT^{2}}\ar[rrr]^{TPm} &  &  & TPT\ar[d]_{T\eta PT}\ar@/^{1pc}/[rrdd]^{\textnormal{id}}\\
 & T^{2}P\ar[rr]^{T\eta TP}\ar[rru]^{T^{2}Pu}\ar[u]_{T\lambda} &  & TPTP\ar[r]_{TPTPu} & TPTPT\ar@/^{0.3pc}/[rd]^{TP\lambda T}\dtwocell[0.6]{d}{TP\omega_{2}T} &  &  & TP^{2}T\ar[rrd]^{T\mu T}\\
T^{2}\ar[r]^{T\eta T}\ar@/_{4pc}/[rrrrrrrrrr]_{\eta T^{2}}\ar@/_{1pc}/[rrrrrdd]_{m}\ar[ur]^{T^{2}\eta} & TPT\ar[rr]^{TPTu}\ar@/_{2pc}/[rrrrrrrr]|-{\textnormal{id}}\ar[rru]_{TPT\eta} &  & TPT^{2}\ar[ur]^{TPT\eta T}\ar@/_{0pc}/[rr]_{TP\eta T^{2}}\ar@/_{1.3pc}/[rrrr]|-{\textnormal{id}} & \;\dtwocell[1]{d}{\omega_{2}T} & TP^{2}T^{2}\ar[rr]^{T\mu T^{2}}\ar@/^{0.3pc}/[rru]^{TP^{2}m} &  & TPT^{2}\ar[rr]^{TPm} &  & TPT\ar[r]^{\lambda T} & PT^{2}\ar[r]^{Pm} & PT\\
 &  &  &  & \;\\
 &  &  &  &  & T\ar@/_{1pc}/[rrrruurr]_{\eta T}
}
}
\]
From here, it is just a matter of restricting the cell $TP\omega_{2}T$
to a $T\omega_{2}$ by naturality. Simplifying, this is the right
hand side of \eqref{W8}.
\end{proof}
\begin{rem}
Whilst one can certainly use a similar version of the above to show
the axiom \eqref{W9} is redundant, doing so is not necessary. This
is because we may construct a compatible extension to the Kleisli
without ever using it.
\end{rem}

\begin{rem}
The decagon formulation includes two coherence conditions which ensure
the extension is compatible with the Kleisli functor $\mathscr{C}\to\mathscr{C}_{P}$.
The first property (DC1) is simply \eqref{W10} or (D5) which we have
already explained is redundant. The second property \eqref{DC2} is
almost \eqref{D4} but without the restriction along the unit $TPTu\colon TPT\to TPT^{2}$. 
\end{rem}

\begin{rem}
\label{redatt} The redundancy of \eqref{W8} almost gives a proof
of the redundancy of axiom \eqref{DC2}. In particular, the left of
\eqref{DC2} is
\[
\resizebox{0.8\textwidth}{!}{\xymatrix{TPT^{2}\ar[d]_{\lambda T^{2}}\ar[r]^{TPT\eta T} & TPTPT\ar[r]^{TP\lambda T}\ar[d]_{\lambda TPT} & TP^{2}T^{2}\ar[r]^{TP^{2}m}\ar[d]_{\lambda PT^{2}} & TP^{2}T\ar[r]^{T\mu T}\ar[d]_{\lambda PT} & TPT\ar[dd]^{\lambda T}\\
PT^{3}\ar[dd]_{PmT}\ar[r]^{PT^{2}\eta T} & PT^{2}PT\ar[dd]_{PmPT}\ar[r]_{PT\lambda T} & PTPT^{2}\ar[d]_{P\lambda T^{2}}\ar[r]^{PTPm} & PTPT\ar[d]_{P\lambda T}\dtwocell[0.45]{r}{\omega_{4}T} & \;\\
 & \dtwocell[0.3]{r}{P\omega_{3}T} & P^{2}T^{3}\ar[d]_{P^{2}mT}\ar[r]^{P^{2}Tm} & P^{2}T^{2}\ar[d]_{P^{2}m}\ar[r]_{\mu T^{2}} & PT^{2}\ar[d]^{Pm}\\
PT^{2}\ar[r]_{PT\eta T}\ar@/_{2pc}/[rr]|-{P\eta T^{2}}\ar@/_{0.5pc}/[rd]_{Pm} & PTPT\ar[r]_{P\lambda T}\dtwocell[0.4]{d}{P\omega_{2}T} & P^{2}T^{2}\ar[r]^{P^{2}m} & P^{2}T\ar[r]^{\mu T} & PT\\
 & PT\ar[urr]_{P\eta T}\ar@/_{0.5pc}/[urrr]_{\textnormal{id}}
}
}
\]
and one can immediately simplify using \eqref{W8}. The issue is \eqref{W7}
would then be required to simplify to the right of \eqref{DC2}, which
likely cannot be shown here.
\end{rem}

\begin{cor}
The assignment $\left(ii\right)$ is a functor and defines an isomorphism.
\end{cor}

\begin{proof}
Since all remaining axioms \eqref{W6} to \eqref{W10} may be deduced
from \eqref{W1} to \eqref{W5}, we know that the remaining condition
\eqref{DC2} must also follow, as we may simplify the diagram of Remark
\ref{redatt}. It follows that $\left(ii\right)$ is a well defined
functor. 

We also point out that from the full decagon form we get the pseudoalgebra
and then no-iteration form above which is a compatible extension to
the Kleisli, and thus axioms \eqref{W3} to \eqref{W5} may be recovered
(because we necessarily get a full pseudodistributive law). Hence
Lemma \ref{ombij} extends to an isomorphism $\left(ii\right)$. 
\end{proof}

\subsection{Equivalence of decagon and pseudoalgebra forms}
\begin{cor}
\label{co3-5} The functor $\left(iii\right)$ and thus also $\left(v\right)$
define equivalences.
\end{cor}

\begin{proof}
Note that the diagram \eqref{main} pseudo-commutes, and so the composite
\[
\xymatrix{\lambda\textnormal{ dec}\myar{\left(iii\right)}{r} & \alpha\textnormal{ ps-alg}\ar[r] & \lambda\textnormal{ dec}}
\]
is isomorphic to the identity, where the second arrow denotes $\left(v\right)$
composed with the other equivalences. The more complex part is checking
that
\[
\xymatrix@R=0.5em{\alpha\textnormal{ ps-alg}\ar[r] & \lambda\textnormal{ dec}\myar{\left(iii\right)}{r} & \alpha\textnormal{ ps-alg}\\
\alpha\ar@{|->}[r] & \alpha\cdot TPu\ar@{|->}[r] & Pm\cdot\alpha T\cdot TPuT
}
\]
is also isomorphic to the identity\footnote{We omit the level of modifications since it is mostly trivial pastings
constructed from the isomorphism data.}. This coherent isomorphism may be constructed directly as the pasting
\begin{equation}
\tag{{G1}}\xymatrix@=1.5em{TPT\ar[d]_{TPuT}\ar[rr]_{TP\eta T}\ar@/^{1.8pc}/[rrrrrrd]^{\textnormal{id}} &  & TP^{2}T\ar[d]_{TPuPT}\ar@/^{0.5pc}/[rrd]^{\textnormal{id}} & \dtwocell[0.5]{dl}{TP\psi}\\
TPT^{2}\ar[d]_{\alpha T}\ar[rr]_{TPT\eta T} &  & TPTPT\ar[rr]^{TP\alpha}\ar[d]_{\alpha PT}\dtwocell[0.5]{rdrrr}{\Psi} &  & TP^{2}T\ar[rr]^{T\mu T} &  & TPT\ar[d]^{\alpha}\\
PT^{2}\ar[rr]_{PT\eta T}\ar[rrd]_{Pm} &  & PTPT\ar[rr]_{P\alpha}\dtwocell[0.5]{d}{P\xi} &  & P^{2}T\ar[rr]_{\mu T} &  & PT\\
 &  & PT\ar[rru]_{P\eta T}\ar@/_{0.5pc}/[rrurr]_{\textnormal{id}}
}
\label{G1}
\end{equation}
which shows that $\left(iii\right)$ is an equivalence, and as $\left(v\right)$
belongs to a pseudo-commuting square where everything else is an equivalence
it must be one itself.
\end{proof}
\begin{rem}
By a similar argument one may show $\alpha$ respects the multiplication
of $T$. More explicitly, we have an isomorphism $Pm\cdot\alpha T\cong\alpha\cdot TPm$
constructed as the pasting
\begin{equation}
\tag{{G2}}\xymatrix@=1.5em{TPT^{2}\ar[dd]_{\alpha T}\ar[rrd]_{TPT\eta T\quad}\ar[rrrr]^{TPm} &  & \dtwocell[0.45]{dr}{TP\xi^{-1}} &  & TPT\ar[d]^{TP\eta T}\ar@/^{0.3pc}/[rrd]^{\textnormal{id}}\\
 &  & TPTPT\ar[rr]^{TP\alpha}\ar[d]_{\alpha PT}\dtwocell[0.5]{drrrr}{\Psi} &  & TP^{2}T\ar[rr]^{T\mu T} &  & TPT\ar[d]^{\alpha}\\
PT^{2}\ar[rr]_{PT\eta T}\ar[rrd]_{Pm} &  & PTPT\ar[rr]_{P\alpha}\dtwocell[0.5]{d}{P\xi} &  & P^{2}T\ar[rr]_{\mu T} &  & PT\\
 &  & PT\ar[rru]_{P\eta T}\ar@/_{0.5pc}/[rrurr]_{\textnormal{id}}
}
\label{G2}
\end{equation}
It is an interesting observation that this construction requires $\xi$
to be invertible, whereas the simplified restriction along $TPuT$
given by \eqref{G1} does not. This means whilst Corollary \ref{co3-5}
may generalize from an equivalence to an adjunction in the skew setting,
the invertibility of $\xi$ would still be required for much of the
theory.
\end{rem}

\subsection{Equivalence of pseudoalgebra, no-iteration and warping forms}

The equivalence of the pseudoalgebra and no-iteration forms relies
on the equivalence of 2-cells in a tricategory and pseudo-pasting
operators, generalizing the result \cite[Lemma 2.2]{noiteration1d}
in dimension one.
\begin{lem}
For $s,t$ and $u$ configured such that $st$ exists and has the
same domain and codomain as $u$, the pseudo-pasting operators of
Definition \ref{pasting}\footnote{It is possible the general definition of pseudo-pasting operator (where
one lacks this configuration) would have extra axioms. Our formulation
is based on what is required in the proof of this Lemma.} are in equivalence with 2-cells $ts\Rightarrow u$.
\end{lem}

\begin{proof}
As in \cite[Lemma 2.2]{noiteration1d} we use the same method of writing
a general $\vartheta^{\#}$ in standard form as $\textnormal{id}_{s}^{\#}\cdot g\cdot t\vartheta$,
so that pasting operators are again determined by their operation
on identities. The only issue here is that the rewriting is only the
same up to isomorphism (since blistering is only an isomorphism),
and to ensure this is an isomorphism of pasting operators we require
the blistering data of $\left(-\right)^{\#}$ to be natural and respect
whiskering and blistering. This is why Definition \ref{pasting} contains
these three respective conditions.
\end{proof}
We now finish the proof of Theorem \ref{pseudodistequiv}, by proving
the equivalences $\left(iv\right)$ and $\left(vi\right)$ of the
pseudoalgebra, no-iteration and warping formulations of a pseudodistributive
law.
\begin{prop}
The pseudoalgebra and no-iteration formulations of a pseudodistributive
law are in equivalence.
\end{prop}

\begin{proof}
Firstly, note that given a no-iteration pseudodistributive law $\left(-\right)^{\lambda}$
we may take $f,g,h$ and $k$ to be identities (and even take $X,Y,Z$
to be the identity on $\mathscr{C}$). From this the data of the pseudoalgebra
version is recovered with $\psi$, $\xi$ and $\Psi$ respectively
given by (where $\alpha$ is the action of $\left(-\right)^{\lambda}$
on an identity)
\[
\xymatrix@=1em{PT\ar@/^{0pc}/[dd]_{uPT}\ar@/^{1pc}/[rrdd]^{\textnormal{id}} &  &  &  &  & T^{2}\ar[dd]_{\textnormal{id}_{T}^{T}=m}\ar@/^{1pc}/[rrdd]^{\left(\eta TY\right)^{\lambda}=\alpha\cdot T\eta T}\\
\; & \dltwocell[0.1]{l}{\psi^{\textnormal{id}}} &  &  &  & \; & \dltwocell[0.1]{l}{\xi^{\textnormal{id}}}\\
TPT\ar[rr]_{\textnormal{id}_{PT}^{\lambda}=\alpha} & \; & PT &  &  & T\myard{\eta T}{rr} &  & PT
}
\]
and
\[
\xymatrix@=1em{TPTPT\ar[dd]_{\textnormal{id}_{PTPT}^{\lambda}=\alpha PT}\ar@/^{1pc}/[rrrrrrrrdd]^{\qquad\qquad\qquad\qquad\left(\left(\textnormal{id}_{PT}^{\lambda}\right)^{P}\right)^{\lambda}=\left(\alpha^{P}\right)^{\lambda}=\left(\mu T\cdot P\alpha\right)^{\lambda}=\alpha\cdot T\mu T\cdot TP\alpha}\\
\; & \; & \dltwocell[0.2]{l}{\Psi^{\textnormal{id},\textnormal{id}}}\\
PTPT\ar[rrrrrrrr]_{\left(\textnormal{id}_{PT}^{\lambda}\right)^{P}=\alpha^{P}=\mu T\cdot P\alpha} &  &  &  &  &  &  &  & PT
}
\]
We have used here the fact $\left(-\right)^{\lambda}$ must be defined
by sending an $f\colon X\to PTY$ to the composite
\[
\xymatrix{TX\myar{Tf}{r} & TPTY\myar{\alpha Y}{r} & PTY}
\]
and a 2-cell $\vartheta\colon f\Rightarrow g\colon X\to PTY$ to the
whiskering $\alpha Y\cdot T\vartheta$, since $\left(-\right)^{\lambda}$
respects whiskering and blistering. Note also that $\psi$, $\xi$
and $\Psi$ become modifications by a similar argument to \cite[Prop. 3.4 and 3.5]{NoIteration}.
Axioms \eqref{I1} and \eqref{I2} respectively become \eqref{M1}
and \eqref{M2} by similarly substituting identities. In this calculation,
it is worth noting that the application of $\left(-\right)^{\lambda}$
to $\psi$ is what gives the $T\psi$ in \eqref{M1}. Moreover, the
non-identity cases such as the $\xi^{uY}$ in \eqref{I1} can be reduced
to identity cases, by identifying the two diagrams
\[
\xymatrix@=1em{TY\ar[rrr]^{\left(\eta TY\cdot uY\right)^{\lambda}}\ar[dd]_{\left(uY\right)^{T}} &  &  & PTY\ar@{=}[dd] &  &  &  & TY\ar[rrr]^{\left(\eta TY\cdot uY\right)^{\lambda}}\ar[d]_{TuY} &  &  & PTY\ar@{=}[d]\\
 & \dtwocell[0.5]{r}{\xi^{uY}} & \; &  &  &  &  & T^{2}Y\ar[d]_{mY}\ar[rrr]^{\left(\eta TY\right)^{\lambda}} & \dtwocell[0.5]{rd}{\xi^{\textnormal{id}}} & \; & PTY\ar@{=}[d]\\
TY\ar[rrr]_{\eta TY} &  &  & PTY &  &  &  & TY\ar[rrr]_{\eta TY} &  & \; & PTY
}
\]
using that the family $\xi^{k}$ respects blistering\footnote{Both $\xi^{uY}$ and $\xi^{\textnormal{id}}$ play a central role
in this paper, as the unit law data $\theta$ in the extended pseudomonad,
and as the data $\xi$ in pseudoalgebra form respectively. These two
pieces of data are combined into a single family $\xi^{k}$ which
cohere via the blistering operation.} and is natural in $k$.

Conversely, given the pseudoalgebra version one will recover the no-iteration
version by defining the pasting operator $\left(-\right)^{\lambda}\colon\mathscr{C}\left(-,PT-\right)\to\mathscr{C}\left(T-,PT-\right)$
by the same formula $f\mapsto\alpha Y\cdot Tf$. It is then easy to
construct $\psi^{f}$ from $\psi$, $\xi^{k}$ from $\xi$, and $\Psi^{g,f}$
from $\Psi$. This is done by taking blisterings to upgrade from identities
to general maps; for example $\xi^{uY}$ is constructed from $\xi^{\textnormal{id}}=\xi$
by $uY$-blistering as in the above diagram. In much the same way,
the coherence axioms \eqref{M1} and \eqref{M2} directly give \eqref{I1}
and \eqref{I2} for the identity case, and are upgraded to general
maps (such as the $g$ in \eqref{I1}) by blistering. By the same
method we see the coherence axioms \eqref{MC1} and \eqref{MC2} give
\eqref{IC1} and \eqref{IC2}.
\end{proof}
\begin{prop}
The no-iteration and warping formulations of a pseudodistributive
law are in isomorphism.
\end{prop}

\begin{proof}
This is simply a matter of restating the axioms on the morphisms of
pseudo-pasting operators $\left(-\right)^{\psi},\left(-\right)^{\xi}$
and $\left(-\right)^{\Psi}$ given in Definition \ref{pswarp} in
terms of the components. The four coherence axioms correspond respectively
to the four of the no-iteration form of Definition \ref{psnoit}.
\end{proof}

\section{Declaration Statement}

\subsection{Author's contribution}

This was entirely the work of the only author.

\subsection{Conflict of interest}

The author declares he has no conflict of interest.

\subsection{Availability of data and materials}

Not applicable.

\subsection{Funding}

The author acknowledges that this work was supported by the Operational
Programme Research, Development and Education Project \textquotedblleft Postdoc@MUNI\textquotedblright{}
(No. CZ.02.2.69/0.0/0.0/18\_053/0016952).

\bibliographystyle{siam}
\bibliography{references}

\begin{thebibliography}{10}

\bibitem{beckdist}
{\sc J.~Beck}, {\em Distributive laws}, in Sem. on {T}riples and {C}ategorical
  {H}omology {T}heory ({ETH}, {Z}\"urich, 1966/67), Springer, Berlin, 1969,
  pp.~119--140.

\bibitem{bisimulation}
{\sc G.~L. Cattani and G.~Winskel}, {\em Profunctors, open maps and
  bisimulation}, Math. Structures Comput. Sci., 15 (2005), pp.~553--614.

\bibitem{cheng2003}
{\sc E.~Cheng, M.~Hyland, and J.~Power}, {\em Pseudo-distributive laws},
  Electronic Notes in Theoretical Computer Science, 83 (2003).

\bibitem{Chk}
{\sc D.~Chikhladze}, {\em {A note on warpings of monoidal structures}}, arXiv
  eprint,  (2015).
\newblock Available at \url{https://arxiv.org/abs/1510.00483}.

\bibitem{relative}
{\sc M.~Fiore, N.~Gambino, M.~Hyland, and G.~Winskel}, {\em Relative
  pseudomonads, {K}leisli bicategories, and substitution monoidal structures},
  Selecta Math. (N.S.), 24 (2018), pp.~2791--2830.

\bibitem{GambinoDL}
{\sc N.~Gambino and G.~Lobbia}, {\em On the formal theory of pseudomonads and
  pseudodistributive laws}, Theory Appl. Categ., 37 (2021), pp.~No. 2, 14--56.

\bibitem{NoiterationMixedAnother}
{\sc E.~R. Hern{\'a}ndez}, {\em {Another characterization of no-iteration
  distributive laws}}, arXiv eprint,  (2020).
\newblock Available at \url{https://arxiv.org/abs/1910.06531}, Version 2.

\bibitem{Hyland2002}
{\sc M.~Hyland, G.~Plotkin, and J.~Power}, {\em Combining Computational
  Effects: Commutativity and Sum}, Springer US, Boston, MA, 2002, pp.~474--484.

\bibitem{kellymaclanecoherence}
{\sc G.~M. Kelly}, {\em On {M}ac{L}ane's conditions for coherence of natural
  associativities, commutativities, etc}, J. Algebra, 1 (1964), pp.~397--402.

\bibitem{kellydist}
\leavevmode\vrule height 2pt depth -1.6pt width 23pt, {\em Coherence theorems
  for lax algebras and for distributive laws}, in Category {S}eminar ({P}roc.
  {S}em., {S}ydney, 1972/1973), 1974, pp.~281--375. Lecture Notes in Math.,
  Vol. 420.

\bibitem{kelly1974}
\leavevmode\vrule height 2pt depth -1.6pt width 23pt, {\em On clubs and
  doctrines}, in Category {S}eminar ({P}roc. {S}em., {S}ydney, 1972/1973),
  Springer, Berlin, 1974, pp.~181--256. Lecture Notes in Math., Vol. 420.

\bibitem{kock1995}
{\sc A.~Kock}, {\em Monads for which structures are adjoint to units}, J. Pure
  Appl. Algebra, 104 (1995), pp.~41--59.

\bibitem{StreetFTM2}
{\sc S.~Lack and R.~Street}, {\em The formal theory of monads. {II}}, J. Pure
  Appl. Algebra, 175 (2002), pp.~243--265.
\newblock Special volume celebrating the 70th birthday of Professor Max Kelly.

\bibitem{mndwarp}
\leavevmode\vrule height 2pt depth -1.6pt width 23pt, {\em On monads and
  warpings}, Cah. Topol. G\'{e}om. Diff\'{e}r. Cat\'{e}g., 55 (2014),
  pp.~244--266.

\bibitem{relativedist}
{\sc G.~Lobbia}, {\em Distributive laws for relative monads}, Appl. Categ.
  Struct., 31 (2023), p.~38.
\newblock Id/No 19.

\bibitem{MacLane}
{\sc S.~MacLane}, {\em Natural associativity and commutativity}, Rice Univ.
  Stud., 49 (1963), pp.~28--46.

\bibitem{MacPare}
{\sc S.~MacLane and R.~Par{\'e}}, {\em Coherence for bicategories and indexed
  categories}, J. Pure Appl. Algebra, 37 (1985), pp.~59--80.

\bibitem{Manes1976}
{\sc E.~G. Manes}, {\em Algebraic theories}, Springer-Verlag, New
  York-Heidelberg, 1976.
\newblock Graduate Texts in Mathematics, No. 26.

\bibitem{marm1997}
{\sc F.~Marmolejo}, {\em Doctrines whose structure forms a fully faithful
  adjoint string}, Theory Appl. Categ., 3 (1997), pp.~No.\ 2, 24--44.

\bibitem{marm1999}
\leavevmode\vrule height 2pt depth -1.6pt width 23pt, {\em Distributive laws
  for pseudomonads}, Theory Appl. Categ., 5 (1999), pp.~No. 5, 91--147.

\bibitem{marm2002}
{\sc F.~Marmolejo, R.~D. Rosebrugh, and R.~J. Wood}, {\em A basic distributive
  law}, J. Pure Appl. Algebra, 168 (2002), pp.~209--226.
\newblock Category theory 1999 (Coimbra).

\bibitem{marm2008}
{\sc F.~Marmolejo and R.~J. Wood}, {\em Coherence for pseudodistributive laws
  revisited}, Theory Appl. Categ., 20 (2008), pp.~No. 5, 74--84.

\bibitem{noiteration1d}
\leavevmode\vrule height 2pt depth -1.6pt width 23pt, {\em Monads as extension
  systems - no iteration is necessary}, Theory Appl. Categ., 24 (2010), pp.~No.
  4, 84--113.

\bibitem{NoIteration}
\leavevmode\vrule height 2pt depth -1.6pt width 23pt, {\em No-iteration
  pseudomonads}, Theory Appl. Categ., 28 (2013), pp.~No. 14, 371--402.

\bibitem{thesisTanaka}
{\sc M.~Tanaka}, {\em Pseudo-Distributive Laws and a Unified Framework for
  Variable Binding}, PhD thesis, University of Edinburgh, 2004.

\bibitem{variablebinding}
{\sc M.~Tanaka and J.~Power}, {\em Pseudo-distributive laws and axiomatics for
  variable binding}, Higher-Order and Symbolic Computation, 19 (2006).

\bibitem{WalkerDL}
{\sc C.~Walker}, {\em Distributive laws via admissibility}, Appl. Categ.
  Structures, 27 (2019), pp.~567--617.

\bibitem{WaltersPHD}
{\sc R.~F.~C. Walters}, {\em A categorical approach to universal algebra}, PhD
  thesis, Australian National University, 1970.

\bibitem{zober1976}
{\sc V.~Z\"oberlein}, {\em Doctrines on {$2$}-categories}, Math. Z., 148
  (1976), pp.~267--279.

\end{thebibliography}

\end{document}